\documentclass[11pt,a4paper]{amsart}%
\usepackage{amsmath}
\usepackage{amsfonts}
\usepackage{graphicx}
\usepackage{amssymb}
\usepackage{xcolor}%
\providecommand{\U}[1]{\protect\rule{.1in}{.1in}}

\def \N {\mathbb{N}}

\def \bd {\boldsymbol}
\theoremstyle{definition}
\newtheorem{definition}{Definition}[section]

\newtheorem{remark}[definition]{Remark}

\theoremstyle{plain}

\newtheorem{theorem}[definition]{Theorem}
\newtheorem{proposition}[definition]{Proposition}
\newtheorem{lemma}[definition]{Lemma}

\numberwithin{equation}{section}
\begin{document}
\title[Sobolev estimates for KFP operators]{Global Sobolev theory for Kolmogorov-Fokker-Planck operators with coefficients
measurable in time and $VMO$ in space}
\author{Stefano Biagi, Marco Bramanti}

\begin{abstract}
We consider Kolmogorov-Fokker-Planck operators of the form%
\[
\mathcal{L}u=\sum_{i,j=1}^{q}a_{ij}(x,t)u_{x_{i}x_{j}}+\sum_{k,j=1}^{N}%
b_{jk}x_{k}u_{x_{j}}-\partial_{t}u,
\]
with $\left(  x,t\right)  \in\mathbb{R}^{N+1},N\geq q\geq1$. We assume that
$a_{ij}\in L^{\infty}\left(  \mathbb{R}^{N+1}\right)  $, the matrix $\left\{
a_{ij}\right\}  $ is symmetric and uniformly positive on $\mathbb{R}^{q}$, and
the \emph{drift}%
\[
Y=\sum_{k,j=1}^{N}b_{jk}x_{k}\partial_{x_{j}}-\partial_{t}%
\]
has a structure which makes the model operator with constant $a_{ij}$
hypoelliptic, translation invariant w.r.t. a suitable Lie group operation, and
$2$-homogeneus w.r.t. a suitable family of dilations. We also assume that the
coefficients $a_{ij}$ are $VMO$ w.r.t. the space variable, and only bounded
measurable in $t$. We prove, for every $p\in\left(  1,\infty\right)  $, global
Sobolev estimates of the kind:%
\begin{align*}
\left\Vert u\right\Vert _{W_{X}^{2,p}\left(  S_{T}\right)  }\equiv &
\sum_{i,j=1}^{q}\left\Vert u_{x_{i}x_{j}}\right\Vert _{L^{p}\left(
S_{T}\right)  }+\left\Vert Yu\right\Vert _{L^{p}\left(  S_{T}\right)  }%
+\sum_{i=1}^{q}\left\Vert u_{x_{i}}\right\Vert _{L^{p}\left(  S_{T}\right)
}\\
+\left\Vert u\right\Vert _{L^{p}\left(  S_{T}\right)  }  &  \leq c\left\{
\left\Vert \mathcal{L}u\right\Vert _{L^{p}\left(  S_{T}\right)  }+\left\Vert
u\right\Vert _{L^{p}\left(  S_{T}\right)  }\right\}
\end{align*}
with $S_{T}=\mathbb{R}^{N}\times\left(  -\infty,T\right)  $ for any
$T\in(-\infty,+\infty]$. Also, the well-posedness in $W_{X}^{2,p}\left(
\Omega_{T}\right)  $, with $\Omega_{T}=\mathbb{R}^{N}\times\left(  0,T\right)
$ and $T\in\mathbb{R}$, of the Cauchy problem%
\[
\left\{
\begin{tabular}
[c]{ll}%
$\mathcal{L}u=f$ & in $\Omega_{T}$\\
$u\left(  \cdot,0\right)  =g$ & in $\mathbb{R}^{N}$%
\end{tabular}
\ \ \ \right.
\]
is proved, for $f\in L^{p}\left(  \Omega_{T}\right)  ,g\in W_{X}^{2,p}\left(
\mathbb{R}^{N}\right)  $.

\end{abstract}
\maketitle

\section{Introduction and main results\label{sec intro}}

\subsection{The problem and its context\label{sec inizio}}

In this paper we deal with Kolmogorov-Fokker-Planck (KFP, in short) operators
of the form%
\begin{equation}
\mathcal{L}u=\sum_{i,j=1}^{q}a_{ij}(x,t)\partial_{x_{i}x_{j}}^{2}%
u+\sum_{k,j=1}^{N}b_{jk}x_{k}\partial_{x_{j}}u-\partial_{t}u,\qquad
(x,t)\in\mathbb{R}^{N+1} \label{L}%
\end{equation}
where $N\geq q\geq1$. The first-order part of the operator, also called
\emph{the drift term}, will be briefly denoted by%
\begin{equation}
Yu=\sum_{k,j=1}^{N}b_{jk}x_{k}\partial_{x_{j}}u-\partial_{t}u.
\label{eq:driftY}%
\end{equation}
We will make the following assumptions:

\begin{itemize}
\item[\textbf{(H1)}] $A_{0}(x,t)=(a_{ij}(x,t))_{i,j=1}^{q}$ is a symmetric
uniformly positive matrix on $\mathbb{R}^{q}$ of coefficients defined in
$\mathbb{R}^{N+1}$ and belonging to $L^{\infty}\left(  \mathbb{R}%
^{N+1}\right)  $, so that
\begin{equation}
\nu|\xi|^{2}\leq\sum_{i,j=1}^{q}a_{ij}(x,t)\xi_{i}\xi_{j}\leq\nu^{-1}|\xi|^{2}
\label{nu}%
\end{equation}
for some constant $\nu>0$, every $\xi\in\mathbb{R}^{q}$, a.e. $\left(
x,t\right)  \in\mathbb{R}^{N+1}$.

The co\-ef\-fi\-cients will be also assumed $VMO$ w.r.t.\thinspace$x$ (and
merely measurable w.r.t.\thinspace$t$) in a sense that will be made precise
later \vspace*{0.1cm}(see Assumption (H3) in Section \ref{sec assumptions}%
).\medskip

\item[\textbf{(H2)}] The matrix $B=(b_{ij})_{i,j=1}^{N}$ satisfies the
following condition: for $m_{0}=q$ and suitable positive integers $m_{1}%
,\dots,m_{k}$ such that
\begin{equation}
m_{0}\geq m_{1}\geq\ldots\geq m_{k}\geq1\quad\mathrm{and}\quad m_{0}%
+m_{1}+\ldots+m_{k}=N, \label{m-cond}%
\end{equation}
we have
\begin{equation}
B=%
\begin{pmatrix}
\mathbb{O} & \mathbb{O} & \ldots & \mathbb{O} & \mathbb{O}\\
B_{1} & \mathbb{O} & \ldots & \ldots & \ldots\\
\mathbb{O} & B_{2} & \ldots & \mathbb{O} & \mathbb{O}\\
\vdots & \vdots & \ddots & \vdots & \vdots\\
\mathbb{O} & \mathbb{O} & \ldots & B_{k} & \mathbb{O}%
\end{pmatrix}
\label{B}%
\end{equation}
where $\mathbb{O}$ is the null matrix, and every block $B_{j}$ is an
$m_{j}\times m_{j-1}$ matrix of rank $m_{j}$ (for $j=1,2,\ldots,k$).
\end{itemize}

\medskip

\noindent We explicitly note that, when $q<N$, the operator $\mathcal{L}$ is
\emph{ultraparabolic}; nevertheless, under assumptions \textbf{(H1)}%
-\textbf{(H2)}, as proved by Lanconelli and Polidoro \cite{LP}, the
\emph{model operator} $\mathcal{L}_{0}$ corresponding to the case of
\emph{constant }$a_{ij}$ (as well as its formal transpose $\mathcal{L}%
_{0}^{\ast}$) is hypoelliptic and satisfies the following properties:

\begin{itemize}
\item[(a)] $\mathcal{L}_{0}$ is left-invariant on the \emph{non-commutative
Lie group} $\mathbb{G}=(\mathbb{R}^{N+1},\circ)$, where the composition law
$\circ$ is defined as follows
\begin{align}
(y,s)\circ(x,t)  &  =(x+E(t)y,t+s)\label{traslazioni}\\
(y,s)^{-1}  &  =(-E(-s)y,-s),\nonumber
\end{align}
and $E(t)=\exp(-tB)$ (which is defined for every $t\in\mathbb{R}$ since the
matrix $B$ is nilpotent). For a future reference, we explicitly notice that
\begin{equation}
(y,s)^{-1}\circ(x,t)=(x-E(t-s)y,t-s), \label{eq:convolutionG}%
\end{equation}
and that the Lebesgue measure is the Haar measure, which is also invariant
with respect to the inversion. \vspace*{0.1cm}

\item[(b)] $\mathcal{L}_{0}$ is homogeneous of degree $2$ with respect to a
nonisotropic family of \emph{dilations} in $\mathbb{R}^{N+1}$, which are
automorphisms of $\mathbb{G}$ and are defined by
\begin{equation}
D(\lambda)(x,t)\equiv(D_{0}(\lambda)(x),\lambda^{2}t)=(\lambda^{q_{1}}%
x_{1},\ldots,\lambda^{q_{N}}x_{N},\lambda^{2}t), \label{dilations}%
\end{equation}
where the $N$-tuple $(q_{1},\ldots,q_{N})$ is given by
\[
(q_{1},\ldots,q_{N})=(\underbrace{1,\ldots,1}_{m_{0}},\,\underbrace{3,\ldots
,3}_{m_{1}},\ldots,\underbrace{2k+1,\ldots,2k+1}_{m_{k}}).
\]
The integer%
\begin{equation}
Q=\sum_{i=1}^{N}q_{i}>N \label{eq:defQhomdim}%
\end{equation}
is called the \emph{homogeneous dimension} of $\mathbb{R}^{N}$, while $Q+2$ is
the homogeneous dimension of $\mathbb{R}^{N+1}$.
\end{itemize}

\bigskip

Actually, Lanconelli-Polidoro in \cite{LP} have studied constant-coefficients
KFP operators corresponding to a wider class of matrices $B$, which are not
nilpotent; these more general operators are hypoelliptic, left-invariant with
respect to the above operation $\circ$, but they are not necessarily
homogeneous. After the seminal paper \cite{LP}, more general families of
degenerate KFP operators of the kind (\ref{L}), satisfying the same structural
conditions on the matrices $A_{0}$ and $B$ but with variable coefficients
$a_{ij}\left(  x,t\right)  $, have been studied by several authors. In
particular, Schauder estimates have been investigated first by Manfredini
\cite{Ma} and later, under more general assumptions on $B$, by Di
Francesco-Polidoro in \cite{DP}, on bounded domains, assuming the coefficients
$a_{ij}$ H\"{o}lder continuous with respect to the intrinsic distance induced
in $\mathbb{R}^{N+1}$ by the vector fields $\partial_{x_{1}},...,\partial
_{x_{q}},Y$. In the framework of Sobolev spaces, local $L^{p}$ estimates for
the derivatives $\partial_{x_{i}x_{j}}^{2}u$ ($i,j=1,2,...,q$) and $Yu$ in
terms of $\mathcal{L}u$ have been established, for operators (\ref{L}) with
$VMO$ coefficients $a_{ij}\left(  x,t\right)  $, by Bramanti, Cerutti,
Manfredini \cite{BCM}. We recall that the $VMO$ assumption allows for some
kind of discontinuity.

In recent years, there has been growing interest, especially motivated from
the research in the field of stochastic differential equations (see e.g.
\cite{PP}), in the study of KFP operators with coefficients $a_{ij}$ rough in
$t$ (say, $L^{\infty}$), and with some mild regularity (for instance,
H\"{o}lder continuity) only w.r.t. the space variables. The Schauder estimates
that one can reasonably expect under this mild assumption consist in
controlling the H\"{o}lder seminorms w.r.t. $x$ of the derivatives involved in
the equations, uniformly in time. Such estimates are sometimes called
\textquotedblleft partial Schauder estimates\textquotedblright. In the paper
\cite{BB Schauder} we have established global partial Schauder estimates for
degenerate KFP operators (\ref{L}) satisfying assumptions \textbf{(H1)}%
-\textbf{(H2)}, with coefficients $a_{ij}$ H\"{o}lder continuous in space,
bounded measurable in time. We have also shown that the second order
derivatives $\partial_{x_{i}x_{j}}^{2}u$ (for $i,j=1,2,...,q$) are actually
locally H\"{o}lder continuous also w.r.t. time. In \cite{BB domain} we have
proved analogous results on bounded cylinders, while more general global
results in the context of partially Dini continuous functions have been
established by Biagi, Bramanti, Stroffolini in \cite{BBS}. Partial Schauder
estimates for degenerate KFP operators have been proved also in \cite{CRHM} by
Chaudru de Raynal, Honor\'{e}, Menozzi, with different techniques and without
getting the H\"{o}lder control in time of second order derivatives. We refer
to the introduction of \cite{BB Schauder} for further references on both
standard and partial Schauder estimates for these operators.

On the other hand, Sobolev estimates have been proved by Menozzi \cite{M} for
a family of KFP operators with very general drift (containing as a special
case the class (\ref{L})), assuming the coefficients $a_{ij}\left(
x,t\right)  $ uniformly continuous in $x$ and bounded measurable in $t$.
Later, Dong and Yastrzhembskiy in \cite{DY}, have proved Sobolev estimates of
this kind for a class of operators (\ref{L}) with the drift of special type
(\textquotedblleft kinetic KFP operator\textquotedblright), assuming the
coefficients $a_{ij}\left(  x,t\right)  $ to be $VMO$ w.r.t. $x$ and bounded
measurable w.r.t. time. We also quote the paper \cite{DY2}, by the same
Authors, where similar results are proved in the case of KFP operators in
divergence form.

In the present paper we prove global Sobolev estimates, on $\mathbb{R}^{N+1}$
and on infinite strips $S_{T}=\mathbb{R}^{N}\times\left(  -\infty,T\right)  $,
for KFP operators (\ref{L}) satisfying assumptions \textbf{(H1)}%
-\textbf{(H2)}, with coefficients $a_{ij}\left(  x,t\right)  $ $VMO$ w.r.t.
$x$ and $L^{\infty}$ w.r.t. time (see assumption (\textbf{H3}) in Section
\ref{sec assumptions} and Theorems \ref{Thm main a priori estimates} and
\ref{Thm global strip} for the precise results). Moreover, we prove the unique
solvability in $W_{X}^{2,p}\left(  S_{T}\right)  $ of the equation
\[
\mathcal{L}u-\lambda u=f\text{ in }S_{T}%
\]
for $\lambda>0$ large enough and $f\in L^{p}\left(  S_{T}\right)
,T\in(-\infty,+\infty]$ (see Theorem \ref{thm:SolvabilityKlambda}). From this
fact we also deduce well-posedness results for the Cauchy problem for
$\mathcal{L}$ on strips $\Omega_{T}=\mathbb{R}^{N}\times\left(  0,T\right)
,T\in\left(  0,+\infty\right)  ,$%
\[
\left\{
\begin{array}
[c]{l}%
\mathcal{L}u=f\text{ in }\Omega_{T}\\
u\left(  \cdot,0\right)  =g
\end{array}
\right.
\]
with $f\in L^{p}\left(  \Omega_{T}\right)  ,g\in W_{X}^{2,p}\left(
\mathbb{R}^{N}\right)  $ (see Theorem \ref{thm:ExistenceCauchy}).

Our class of operators contains that one studied in \cite{DY}, but our
technique is completely different. Actually, here we apply the technique that
we have recently introduced in \cite{BB VMO}, which is based on a combination
of different ingredients. To describe this technique and the strategy of this
paper, let us briefly recall two different approaches which have been followed
so far to handle uniformly elliptic or parabolic nonvariational operators with
$VMO$ coefficients. The first one dates back to the early 1990s with the works
of Chiarenza-Frasca-Longo \cite{CFL1}, \cite{CFL2} for uniformly elliptic
operators. That technique heavily relies on representation formulas of
$\partial_{x_{i}x_{j}}^{2}u$ in terms of $\mathcal{L}u$, via a singular kernel
\textquotedblleft with variable coefficients\textquotedblright, obtained from
the fundamental solution of the model operator with constant coefficients.
This singular kernel needs to be expanded in series of spherical harmonics,
getting singular kernels \textquotedblleft with constant
coefficients\textquotedblright\ (i.e., of convolution type). One then needs to
apply results of $L^{p}$ continuity of singular integral operators and of the
\emph{commutator} of a singular integral operator with a $BMO$ function. This
technique, and its furher extensions to parabolic operators (see \cite{BC}),
KFP operators (see \cite{BCM}), and nonvariational operators structured on
H\"{o}rmander vector fields (see \cite{BB1}, \cite{BB2}), exploits both
translation invariance and homogeneity (in suitable senses) of the model
operator with constant $a_{ij}$ or, in some cases, the possibility of
approximating locally the operator under study with another one possessing
these properties. Moreover, it relies on the knowledge of fine properties of
the fundamental solution of the constant coefficient operator, with bounds on
the derivatives of every order of this fundamental solution, which have to be
uniform as the constant matrix $\left(  a_{ij}\right)  $ ranges in a fixed
ellipticity class.

A different technique to prove $W^{2,p}$ a-priori estimates for nonvariational
elliptic or parabolic operators with $VMO$ coefficients has been devised in
2007 by Krylov \cite{K}, and exploited in a series of subsequent papers. A key
step in Krylov' technique constists in establishing a pointwise estimate on
the sharp maximal function of the second derivatives of $u$ in terms of
$\overline{\mathcal{L}}u$, where $\overline{\mathcal{L}}$ is the model
operator with constant $a_{ij}$ (if the final goal is to study elliptic
operators with $VMO$ coefficients)\ or coefficients $a_{ij}\left(  t\right)
\in L^{\infty}\left(  \mathbb{R}\right)  $ (if the final goal is to study
parabolic operators with coefficients $a_{ij}\left(  x,t\right)  $ $VMO$
w.r.t. $x$ and $L^{\infty}$ w.r.t. $t$). Once this estimate is proved, a
clever procedure allows to exploit the $VMO$ assumption on $a_{ij}$ for
replacing the model operator $\overline{\mathcal{L}}$ with $\mathcal{L}$. In
turn, in order to establish this estimate on the sharp maximal function of
$\partial_{x_{i}x_{j}}^{2}u$, Krylov makes use of many results on elliptic or
parabolic operators with constant coefficients (or, in the parabolic case,
coefficients only depending on $t$), including several classical solvability
results, together with pointwise estimates on the derivatives of the
solutions, and exploits translations, dilations, and the Poincar\'{e}'s
inequality in the space variables. The extension of \emph{this part }of
Krylov' technique to degenerate operators of H\"{o}rmander type seems very
difficult. So far, we can quote only the paper \cite{BT} by Bramanti-Toschi
where this technique has been successfully implemented, for novariational
operators modeled on H\"{o}rmander vector fields in Carnot groups.

In \cite{BB VMO} we have introduced a new approach which combines some ideas
of both the strategies described above (by Chiarenza-Frasca-Longo and by
Krylov), which we are now going to describe. In the present situation, our
\emph{model operator }is the KFP operator with coefficients \emph{only
depending on time }(in a merely $L^{\infty}$ way):%
\begin{equation}
\overline{\mathcal{L}}u=\sum_{i,j=1}^{q}a_{ij}(t)\partial_{x_{i}x_{j}}%
^{2}u+\sum_{k,j=1}^{N}b_{jk}x_{k}\partial_{x_{j}}u-\partial_{t}u.
\label{L model t}%
\end{equation}
We follow Krylov in the idea of exploiting an estimate on the sharp maximal
function of $\partial_{x_{i}x_{j}}^{2}u$ (in terms of $\overline{\mathcal{L}%
}u$) to prove $L^{p}$ estimates for the operator $\mathcal{L}$ with
coefficients $VMO$ w.r.t.\,$x$ and $L^{\infty}$ w.r.t.\,$t$\ (see our Theorem
\ref{Thm Krylov main step}). In order to prove this bound on the sharp maximal
function, however, we do not follow Krylov' technique but use representation
formulas for second order derivatives in terms of the model operator
$\overline{\mathcal{L}}$. Since an explicit fundamental solution for
$\overline{\mathcal{L}}$ has been built in \cite{BP}, the singular kernel
appearing in this representation formula, although not of convolution type,
can be directly studied thanks to the theory of singular integrals in spaces
of homogeneous type. The link between the two ingredients of our proof
(singular integral operators \emph{and }sharp maximal function) is contained
in an abstract real analysis result in spaces of homogeneous type which has
been proved in \cite{BB VMO}. By the way, we note that the lack of
homogeneity, and the lack of regularity w.r.t.\,the $t$ variable, of the
fundamental solution of $\overline{\mathcal{L}}$, prevents us from applying
Chiarenza-Frasca-Longo's technique in this situation.

As usual for operators with variable coefficients, we first prove a-priori
estimates for smooth functions with \emph{small }support. Passing from this
result to global a-priori estimates for any $W_{X}^{2,p}$-function requires
the use of suitable cutoff functions, a covering theorem, and interpolation
inequalities for first order derivatives. In the present setting, cutoff
functions are easily built thanks to the presence of translations and
dilations. A general covering theorem in spaces of homogeneous type, proved in
\cite{BB VMO}, applies also to our situation. Finally, in the present
situation the required interpolation inequality is just the Euclidean one on
$\partial_{x_{i}}$.

\subsection{Structure of the paper}

In Section \ref{sec assumptions} we introduce the geometric structure related
to our operator (translations, dilations, distance), and point out some known
facts about these structures. We also introduce some real variable notions
related to maximal functions, and recall the related known inequalities. We
then introduce the function spaces that will be used throughout the paper,
make our assumptions and state our main results. In Section
\ref{sec operators t} we study the model operator $\overline{\mathcal{L}}$
(\ref{L model t}) with coefficients depending only on $t$. We recall the known
facts about its fundamental solution $\Gamma$, prove suitable representation
formulas, prove the $L^{p}$ continuity of the singular integral operator
corresponding to the singular kernel $\partial_{x_{i}x_{j}}^{2}\Gamma$, and
with these tools we prove Theorem \ref{Thm Krylov main step}, which gives a
bound on the oscillation of $\partial_{x_{i}x_{j}}^{2}u$ in terms of
$\overline{\mathcal{L}}u$. In Section \ref{Sec operators a(x,t)} we consider
operators with coefficients $a_{ij}\left(  x,t\right)  $ satisfying our $VMO$
assumption in the space variables. Applying Theorem \ref{Thm Krylov main step}%
, we prove an analogous bound on the oscillation of $\partial_{x_{i}x_{j}}%
^{2}u$ in terms of $\mathcal{L}u$ (Section \ref{subsec mean oscillation}%
).\ This bound implies, via the maximal inequalities and exploiting our $VMO$
assumption, an $L^{p}$ estimate on $\partial_{x_{i}x_{j}}^{2}u$ in terms of
$\mathcal{L}u$, for smooth functions with compact support (Section
\ref{subsec Lp}). Finally, we extend the $L^{p}$ bound to a global Sobolev
estimate, with the techiques already sketched at the end of Section
\ref{sec inizio}. In Section \ref{Sec existence} we come to the existence
results. First (Section \ref{subsec refined}) we show that the global
estimates proved in Section \ref{subsec Lp} still hold on strips
$S_{T}=\mathbb{R}^{N}\times\left(  -\infty,T\right)  $ for every
$T\in\mathbb{R}$, and prove the well-posedness of the equation%
\[
\mathcal{L}u-\lambda u=f\text{ in }S_{T}%
\]
for $\lambda>0$ large enough and $f\in L^{p}\left(  S_{T}\right)  $ with
$T\in(-\infty,+\infty]$. Then, in Section \ref{Subsec Cauchy}, we deduce a
well-posedness result for the Cauchy problem for $\mathcal{L}$.

\subsection{Assumptions and main results\label{sec assumptions}}

We can now start giving some precise definitions which will allow to state our
main result.

Throughout the paper, points of $\mathbb{R}^{N+1}$ will be sometimes denoted
by the compact notation%
\[
\xi=(x,t),\quad\eta=(y,s).
\]
Let us introduce the metric structure related to the operator $\mathcal{L}$
that will be used throughout the following. The vector fields $\partial
_{x_{1}},\ldots,\partial_{x_{q}},Y$ form a system of H\"{o}rmander vector
fields in $\mathbb{R}^{N+1}$, left-invariant w.r.t.\thinspace the
com\-po\-si\-tion law $\circ$ defined in (\ref{traslazioni}). The vector
fields $\partial_{x_{i}}$ ($i=1,...,q$) are homogeneous of degree $1$, while
$Y$ is homogeneous of degree $2$ w.r.t.\thinspace the dilations $D(\lambda)$
defined in (\ref{dilations}). As every set of H\"{o}rmander vector fields with
drift, this system
\[
X=\left\{  \partial_{x_{1}},\ldots,\partial_{x_{q}},Y\right\}
\]
induces a (weighted) control distance $d_{X}$ in $\mathbb{R}^{N+1}$; we now
review this definition in our special case. First of all, given $\xi
=(x,t),\,\eta=(y,s)\in\mathbb{R}^{N+1}$ and $\delta>0$, we denote by
$C_{\xi,\eta}(\delta)$ the class of \emph{absolutely continuous} curves
\[
\varphi:[0,1]\longrightarrow\mathbb{R}^{N+1}%
\]
which satisfy the following properties:

\begin{itemize}
\item[(i)] $\varphi(0) = \xi$ and $\varphi(1) = \eta$; \vspace*{0.1cm}

\item[(ii)] for almost every $t\in\lbrack0,1]$ one has
\[
\varphi^{\prime}(t)=\sum_{i=1}^{q}a_{i}(t)\varphi_{i}(t)+a_{0}(t)Y_{\varphi
(t)},
\]
where $a_{0},\ldots,a_{q}:[0,1]\rightarrow\mathbb{R}$ are measurable functions
such that
\[
\left\vert \text{$a_{i}(t)$}\right\vert \leq\text{$\delta$ (for $i=1,\ldots
,q$)\quad and \quad$\left\vert a_{0}(t)\right\vert \leq\delta^{2}$}%
\qquad\text{a.e. on $[0,1]$}.
\]

\end{itemize}

Here $\varphi_{i}$ are the components of the vector function $\varphi$ and
$Y_{\varphi(t)}$ stands for the vector field $Y$ evaluated at the point
$\varphi(t)$. We then define
\[
d_{X}(\xi,\eta)=\inf\left\{  \delta>0:\,\exists\,\,\varphi\in C_{\xi,\eta
}(\delta)\right\}  .
\]
Since $\partial_{x_{1}},\ldots,\partial_{x_{q}},Y$ satisfy H\"{o}rmander's
rank condition, it is well-known that the function $d_{X}$ is a distance in
$\mathbb{R}^{N+1}$ (see, e.g., \cite[Prop.\,1.1]{NSW} or \cite[Chap.1]%
{BBbook}); in particular, for every fixed $\xi,\eta\in\mathbb{R}^{N+1}$ there
always exists $\delta>0$ such that $C_{\xi,\eta}(\delta)\neq\varnothing$. In
addition, by the invariance/homogeneity properties of the vector fields, we
see that

\begin{itemize}
\item[(a)] $d_{X}$ is left-invariant with respect to $\circ$, that is,
\begin{equation}
d_{X}(\xi,\eta)=d_{X}(\eta^{-1}\circ\xi,0) \label{d_X}%
\end{equation}

\item[(b)] $d_{X}$ is is jointly $1$-homogeneous with respect to $D(\lambda)$,
that is
\begin{equation}
d_{X}(D(\lambda)\xi,D(\lambda)\eta)=\lambda d_{X}(\xi,\eta)\qquad\text{ for
every $\lambda>0$}. \label{rho_X}%
\end{equation}

\end{itemize}

As a consequence of (\ref{d_X}), the function $\rho_{X}(\xi):=d_{X}(\xi,0)$ satisfies

\begin{enumerate}
\item $\rho_{X}(\xi^{-1})=\rho_{X}(\xi)$; \vspace*{0.05cm}

\item $\rho_{X}(\xi\circ\eta)\leq\rho_{X}(\xi)+\rho_{X}(\eta)$;
\end{enumerate}

moreover, by (\ref{rho_X}) we also have

\begin{itemize}
\item[(1)'] $\rho_{X}(\xi)\geq0$ and $\rho_{X}(\xi)=0\,\Leftrightarrow\,\xi
=0$; \vspace*{0.05cm}

\item[(2)'] $\rho_{X}(D(\lambda)\xi)=\lambda\rho_{X}(\xi)$,
\end{itemize}

and this means that $\rho_{X}$ is a \emph{homogeneous norm} in $\mathbb{R}%
^{N+1}$. \medskip

We now observe that also the function
\begin{equation}
\rho(\xi)=\rho(x,t):=\Vert x\Vert+\sqrt{|t|}=\sum_{i=1}^{N}|x_{i}|^{1/q_{i}%
}+\sqrt{|t|} \label{eq:defrhonorm}%
\end{equation}
is a homogeneous norm in $\mathbb{R}^{N+1}$ (i.e., it satisfies properties
(1)'-(2)' above), and therefore it is \emph{globally equivalent} to $\rho_{X}$
(see e.g. \cite[Thm.3.12]{BBbook}): there exist $c_{1},c_{2}>0$ such that
\[
c_{1}\rho_{X}(\xi)\leq\rho(\xi)\leq c_{2}\rho_{X}(\xi)\qquad\forall\,\,\xi
\in\mathbb{R}^{N+1}.
\]
As a consequence of this fact, the map
\begin{equation}
d(\xi,\eta):=\rho(\eta^{-1}\circ\xi) \label{d}%
\end{equation}
is a left-invariant, $1$-homogeneous, (quasisymmetric) \emph{quasidistance} on
$\mathbb{R}^{N+1}$. This means, precisely, that there exists a `structural
constant' $\boldsymbol{\kappa}>0$ such that
\begin{align}
d(\xi,\eta)  &  \leq\boldsymbol{\kappa}\left\{  d(\xi,\zeta)+d(\eta
,\zeta)\right\}  \qquad\forall\,\,\xi,\eta,\zeta\in\mathbb{R}^{N+1}%
;\label{quasitriangle}\\
d(\xi,\eta)  &  \leq\boldsymbol{\kappa}\,d(\eta,\xi)\qquad\forall\,\,\xi
,\eta\in\mathbb{R}^{N+1}. \label{quasisymmetric}%
\end{align}
The quasidistance $d$ is \emph{globally equivalent} to the control distance
$d_{X}$; hence, we will systematically use this quasidistance $d$ and the
associated balls
\[
B_{r}(\xi):=\left\{  \eta\in\mathbb{R}^{N+1}:\,d(\eta,\xi)<r\right\}
\qquad(\text{for $\xi\in\mathbb{R}^{N+1}$ and $r>0$}).
\]

\begin{remark}
\label{rem:propd} For a future reference, we list below some properties of $d$.

\begin{enumerate}
\item By \eqref{eq:convolutionG}, $d$ has the following explicit expression
\begin{equation}
d(\xi,\eta)=\Vert x-E(t-s)y\Vert+\sqrt{|t-s|}, \label{eq:explicitd}%
\end{equation}
for every $\xi=(x,t),\,\eta=(y,s)\in\mathbb{R}^{N+1}$. \medskip

\item Since $E(0)=\operatorname{Id}$, from \eqref{eq:explicitd} we get
\begin{equation}
d((x,t),(y,t))=\Vert x-y\Vert\qquad\text{for every $x,y\in\mathbb{R}^{N}$ and
$t\in\mathbb{R}$}, \label{d stesso t}%
\end{equation}
from which we derive that the quasidistance $d$ \emph{is sym\-me\-tric when
applied to points with the same $t$-coordinate}. We explicitly emphasize that
an a\-na\-lo\-gous property for points with the same $x$-coordinate \emph{does
not hold}: in fact, for every fixed $x\in\mathbb{R}^{N}$ and $t,s\in
\mathbb{R}$ we have
\[
d((x,t),(x,s))=\Vert x-E(t-s)x\Vert+\sqrt{|t-s|}\neq\sqrt{|t-s|}.
\]

\item Let $\xi\in\mathbb{R}^{N+1}$ be fixed, and let $r>0$. Since $d$
satisfies the quasi-triangular inequality \eqref{quasitriangle}, if $\eta
_{1},\eta_{2}\in B_{r}(\xi)$ we have
\[
d(\eta_{1},\eta_{2})<2\boldsymbol{\kappa}r.
\]

\item Taking into account the very definition of $d$, and bearing in mind that
$\rho$ is a \emph{homogeneous norm} in $\mathbb{R}^{N+1}$, it is readily seen
that
\begin{equation}
B_{r}(\xi)=\xi\circ B_{r}(0)=\xi\circ D\left(  r\right)  \left(
B_{1}(0)\right)  \quad\forall\,\,\xi\in\mathbb{R}^{N+1},\,r>0.
\label{eq:balltraslD}%
\end{equation}
From this, since the Lebesgue measure is a Haar measure on $\mathbb{G}%
=(\mathbb{R}^{N+1},\circ)$, we immediately obtain the following identity
\begin{equation}
\left\vert B_{r}\left(  \xi\right)  \right\vert =\left\vert B_{r}\left(
0\right)  \right\vert =\omega\,r^{Q+2} \label{measure ball}%
\end{equation}
where $\omega:=\left\vert B_{1}\left(  0\right)  \right\vert >0$. Identity
\eqref{measure ball} illustrates the role of $Q+2$ as the \emph{homogeneous
dimension} of $\mathbb{R}^{N+1}$ (w.r.t.\thinspace the dilations $D(\lambda)$).

\item Property (\ref{measure ball}) in particular implies that \emph{a global
doubling condition holds}:%
\begin{equation}
\left\vert B_{2r}\left(  \xi\right)  \right\vert \leq c\left\vert B_{r}\left(
\xi\right)  \right\vert \label{doubling}%
\end{equation}
for some constant $c>0$, every $\xi\in\mathbb{R}^{N+1}$ and $r>0$. In turn,
this fact implies that $\mathbb{R}^{N+1}$, endowed with the quasidistance $d$
(or the equivalent control distance $d_{X}$) and the Lebesgue measure is a
\emph{space of homogeneous type}, in the sense of Coifman-Weiss \cite{CW}.
This will be a key point, in order to apply several deep known results from
real analysis.

\item Another consequence of (\ref{measure ball}) that we will apply in the
following is the inequality%
\begin{equation}
\left\vert B_{kr}\left(  \xi\right)  \right\vert \leq ck^{Q+2}\left\vert
B_{r}\left(  \xi\right)  \right\vert \label{growth balls}%
\end{equation}
for every $\xi\in\mathbb{R}^{N+1},r>0,k\geq1$.
\end{enumerate}
\end{remark}

\bigskip

The quasidistance $d$ allows us to define the function spaces which will be
used throughout the paper.

\begin{definition}
\label{Def VMO}For any $f\in L_{loc}^{1}\left(  \mathbb{R}^{N+1}\right)  $ we
define the (partial) $VMO_{x}$ modulus of $f$ as the function%
\[
\eta_{f}\left(  r\right)  =\sup_{\xi_{0}\in\mathbb{R}^{N+1},\rho\leq r}%
\frac{1}{\left\vert B_{\rho}\left(  \xi_{0}\right)  \right\vert }\int%
_{B_{\rho}\left(  \xi_{0}\right)  }\left\vert f\left(  x,t\right)  -f\left(
\cdot,t\right)  _{B_{\rho}\left(  x_{0},t_{0}\right)  }\right\vert dxdt,
\]
for any $r>0,$ where, throughout the following, we let%
\begin{equation}
f\left(  \cdot,t\right)  _{B}=\frac{1}{\left\vert B\right\vert }\int%
_{B}f\left(  y,t\right)  dyds. \label{partial mean}%
\end{equation}
We say that $f\in BMO_{x}\left(  \mathbb{R}^{N+1}\right)  $ if $\eta_{f}$ is
bounded; we say that $f\in VMO_{x}\left(  \mathbb{R}^{N+1}\right)  $ if,
moreover, $\eta_{f}\left(  r\right)  \rightarrow0$ as $r\rightarrow0^{+}$.
\end{definition}

Note that in (\ref{partial mean}) we compute the integral average, over a ball
of $\mathbb{R}^{N+1}$, of the function $f\left(  \cdot,t\right)  $ which only
depends on the space variable in $\mathbb{R}^{N}$ (since $t$ is fixed).
Nevertheless, due to the nontrivial structure of metric balls, this quantity
cannot be rewritten as an integral average over some fixed subset of
$\mathbb{R}^{N}$. This phenomenon is different from what happens in the
parabolic case dealt in \cite{K}.

Note also that if $f\in L^{\infty}\left(  \mathbb{R}^{N+1}\right)  $ then
obviously $f\in BMO_{x}\left(  \mathbb{R}^{N+1}\right)  $ with $\eta
_{f}\left(  r\right)  \leq2\left\Vert f\right\Vert _{L^{\infty}\left(
\mathbb{R}^{n}\right)  }$. Our last assumption on the variable coefficients
$a_{ij}$ will be the following:

\begin{itemize}
\item[\textbf{(H3)}] We ask that the coefficients $a_{ij}$ in (\ref{L}) belong
to $VMO_{x}\left(  \mathbb{R}^{N+1}\right)  $.
\end{itemize}

Letting, for any $R>0,$
\begin{equation}
a^{\sharp}\left(  R\right)  =\max_{i,j=1,...,q}\eta_{a_{ij}}\left(  R\right)
, \label{mod VMO coeff}%
\end{equation}
our bounds will depend quantitatively on the coefficients through the function
$a^{\sharp}$ and the number $\nu$ in (\ref{nu}).

Let us now introduce the Sobolev spaces related to our system of H\"{o}rmander
vector fields, fixing the related notation (see \cite[Chap.2]{BBbook} for details).

\begin{definition}
[Sobolev spaces]\label{Def Sobolev}Under the above assumptions, for any
$p\in\left[  1,\infty\right]  $, $k=1,2$, and domain $\Omega\subseteq
\mathbb{R}^{N+1}$, we define the Sobolev space%
\[
W_{X}^{k,p}\left(  \Omega\right)  =\left\{  f:\Omega\rightarrow\mathbb{R}%
:\left\Vert f\right\Vert _{W_{X}^{k,p}\left(  \Omega\right)  }<\infty\right\}
\]
where%
\begin{align*}
\left\Vert f\right\Vert _{W_{X}^{1,p}\left(  \Omega\right)  }  &  =\left\Vert
f\right\Vert _{L^{p}\left(  \Omega\right)  }+\sum_{i=1}^{q}\left\Vert
\partial_{x_{i}}f\right\Vert _{L^{p}\left(  \Omega\right)  }\\
\left\Vert f\right\Vert _{W_{X}^{2,p}\left(  \Omega\right)  }  &  =\left\Vert
f\right\Vert _{W_{X}^{1,p}\left(  \Omega\right)  }+\sum_{i,j=1}^{q}\left\Vert
\partial_{x_{i}x_{j}}^{2}f\right\Vert _{L^{p}\left(  \Omega\right)
}+\left\Vert Yf\right\Vert _{L^{p}\left(  \Omega\right)  }%
\end{align*}
and all the derivatives are meant in weak sense.
\end{definition}

We are finally in position to state our first main result.

\begin{theorem}
[Global Sobolev estimates]\label{Thm main a priori estimates}Let $\mathcal{L}$
be an operator as in \eqref{L}, and assume that \emph{\textbf{(H1)},
\textbf{(H2)}, \textbf{(H3)}} hold. Then, for every $p\in\left(
1,\infty\right)  $ there exists a constant $c>0$, depending on $p$, the matrix
$B$ in (\ref{B}), the number $\nu$ in (\ref{nu}), and the function $a^{\#}$ in
(\ref{mod VMO coeff}), such that%
\begin{equation}
\left\Vert u\right\Vert _{W_{X}^{2,p}\left(  \mathbb{R}^{N+1}\right)  }\leq
c\left\{  \Vert\mathcal{L}u\Vert_{L^{p}\left(  \mathbb{R}^{N+1}\right)
}+\Vert u\Vert_{L^{p}\left(  \mathbb{R}^{N+1}\right)  }\right\}
\label{eq:globalW2pMain}%
\end{equation}
{for every function }$u\in W_{X}^{2,p}\left(  \mathbb{R}^{N+1}\right)  $.
\end{theorem}

Our second main result is the well-posedness of the Cauchy problem for
$\mathcal{L}$, that is,
\begin{equation}%
\begin{cases}
\mathcal{L}u=f & \text{in $\Omega_{T}\equiv\mathbb{R}^{N}\times(0,T)$}\\
u(\cdot,0)=g & \text{in $\mathbb{R}^{N},$}%
\end{cases}
\label{eq:PbCauchyL}%
\end{equation}
where $T\in\left(  0,+\infty\right)  $ is fixed, $f\in L^{p}(\Omega_{T})$ and
$g\in W_{X}^{2,p}\left(  \mathbb{R}^{N}\right)  $. By the way, saying that a
function $g$ depending on $x$ alone belongs to $W_{X}^{2,p}\left(
\mathbb{R}^{N}\right)  $, obviously means that it belongs to $W_{X}%
^{2,p}\left(  \Omega_{T}\right)  $ if we regard $g$ as a function of $\left(
x,t\right)  $. We the obvious meaning of norms, we have%
\begin{equation}
\Vert g\Vert_{W_{X}^{2,p}(\Omega_{T})}=T^{\frac{1}{p}}\Vert g\Vert
_{W_{X}^{2,p}(\mathbb{R}^{N})}. \label{g spazio e tempo}%
\end{equation}

In order to give sense to the initial condition in (\ref{eq:PbCauchyL})
avoiding a delicate notion of trace, following \cite{KrylovBook}, let us
introduce in an indirect way the Sobolev spaces of $W_{X}^{2,p}\left(
\Omega_{T}\right)  $ functions \textquotedblleft vanishing for $t=0$%
\textquotedblright:

\begin{definition}
We say that $u\in\mathring{W}_{X}^{2,p}\left(  \Omega_{T}\right)  $ if $u\in
W_{X}^{2,p}\left(  \Omega_{T}\right)  $ and the function $\widetilde{u}$,
obtained from $u$ extending it as zero for $t<0$, belongs to $W_{X}%
^{2,p}\left(  S_{T}\right)  $. Functions in $\mathring{W}_{X}^{2,p}\left(
\Omega_{T}\right)  $ will be thought as defined either on $\Omega_{T}$ or on
$S_{T}$ (vanishing for $t<0$).
\end{definition}

We can now give the precise definition of \emph{solution} to the Cauchy problem.

\begin{definition}
\label{def:SolCauchypb}Let $f\in L^{p}(\Omega_{T})$ and $g\in W_{X}%
^{2,p}\left(  \mathbb{R}^{N}\right)  $ for some $p\in\left(  1,\infty\right)
$ and $T\in\left(  0,+\infty\right)  $. We say that a function $u\in
W_{X}^{2,p}(\Omega_{T})$ is a solution to problem (\ref{eq:PbCauchyL}) if

\begin{enumerate}
\item $\mathcal{L}u=f$ a.e.\thinspace\ in $\Omega_{T}$;

\item $u-g\in\mathring{W}_{X}^{2,p}\left(  \Omega_{T}\right)  $.
\end{enumerate}
\end{definition}

Then, our second main result is the following:

\begin{theorem}
[Well-posedness of the Cauchy problem]\label{thm:ExistenceCauchy}Under the
same assumptions on $\mathcal{L}$ of Theorem \ref{Thm main a priori estimates}%
, let $f\in L^{p}(\Omega_{T})$ and $g\in W_{X}^{2,p}\left(  \mathbb{R}%
^{N}\right)  $ for some $p\in\left(  1,\infty\right)  $ and $T\in\left(
0,+\infty\right)  $. Then, there exists a unique solution $u\in W_{X}%
^{2,p}(\Omega_{T})$ of the Cauchy problem (\ref{eq:PbCauchyL}). Moreover, the
following estimate holds%
\begin{equation}
\Vert u\Vert_{W_{X}^{2,p}(\Omega_{T})}\leq c\,\left\{  \left\Vert f\right\Vert
_{L^{p}(\Omega_{T})}+\Vert g\Vert_{W_{X}^{2,p}(\mathbb{R}^{N})}\right\}  ,
\label{eq:estimateSolCauchy}%
\end{equation}
where $c$ depends on $p$, $T$, the matrix $B$ in \eqref{B}, the number $\nu$
in \eqref{nu}, and the function $a^{\#}$ in \eqref{mod VMO coeff}.
\end{theorem}

We end this Section introducing some more known facts from real analysis in
spaces of homogeneous type. We will use in the following two different kinds
of \emph{maximal functions}.

\begin{definition}
\label{Def Maximal}For $f\in L_{loc}^{1}\left(  \mathbb{R}^{N+1}\right)  $,
$\xi\in\mathbb{R}^{N+1}$, we define the \emph{Hardy-Littlewood}
(uncentered)\emph{ maximal function} as:
\begin{equation}
\mathcal{M}f\left(  \xi\right)  =\sup_{\substack{B_{r}\left(  \overline{\xi
}\right)  \ni\xi\\\overline{\xi}\in\mathbb{R}^{N+1},r>0}}\frac{1}{\left\vert
B_{r}\left(  \overline{\xi}\right)  \right\vert }\int_{B_{r}\left(
\overline{\xi}\right)  }\left\vert f\left(  \eta\right)  \right\vert d\eta.
\label{maximal HL}%
\end{equation}

\end{definition}

Since $\left(  \mathbb{R}^{N+1},d,\left\vert \cdot\right\vert \right)  $ is a
space of homogeneous type, by \cite[Thm.2.1]{CW}, we have:

\begin{theorem}
\label{Thm Maximal}For every $p\in(1,\infty]$ there exists $c>0$ such that,
for every $f\in L^{p}\left(  \mathbb{R}^{N+1}\right)  $,%
\begin{equation}
\left\Vert \mathcal{M}f\right\Vert _{L^{p}\left(  \mathbb{R}^{N+1}\right)
}\leq c_{p}\left\Vert f\right\Vert _{L^{p}\left(  \mathbb{R}^{N+1}\right)  }.
\label{HL ineq}%
\end{equation}

\end{theorem}

Another kind of maximal function, which can also be introduced in any space of
homogeneous type, is the following:

\begin{definition}
\label{Def sharp}For $f\in L_{loc}^{1}\left(  \mathbb{R}^{N+1}\right)  $,
$\xi\in\mathbb{R}^{N+1}$, we define the \emph{sharp maximal function }of $f$
as:%
\begin{equation}
f^{\#}\left(  \xi\right)  =\sup_{\substack{B_{r}\left(  \overline{\xi}\right)
\ni\xi\\\overline{\xi}\in\mathbb{R}^{N+1},r>0}}\frac{1}{\left\vert
B_{r}\left(  \overline{\xi}\right)  \right\vert }\int_{B_{r}\left(
\overline{\xi}\right)  }\left\vert f\left(  \eta\right)  -f_{B_{r}\left(
\overline{\xi}\right)  }\right\vert d\eta, \label{sharp maximal}%
\end{equation}
where%
\[
f_{B_{r}(\overline{\xi})}=\frac{1}{\left\vert B_{r}\left(  \overline{\xi
}\right)  \right\vert }\int_{B_{r}\left(  \overline{\xi}\right)  }f\left(
\eta\right)  d\eta.
\]

\end{definition}

Applying to our context the result proved in \cite[Prop.3.4]{PS} in the
general setting of spaces of homogeneous type of infinite measure, we have the
following result, generalizing the well-known Fefferman-Stein inequality which
holds in Euclidean spaces:

\begin{theorem}
\label{Thm Fefferman Stein}For every $p\in\left(  1,\infty\right)  $ there
exists $C_{p}$ (depending on $p$ and the doubling constant in (\ref{doubling}%
)) such that for every $f\in L^{\infty}\left(  \mathbb{R}^{N+1}\right)  $, $f$
with bounded support, we have%
\begin{equation}
\left\Vert f\right\Vert _{L^{p}\left(  \mathbb{R}^{N+1}\right)  }\leq
C_{p}\left\Vert f^{\#}\right\Vert _{L^{p}\left(  \mathbb{R}^{N+1}\right)  }.
\label{Fefferman Stein}%
\end{equation}

\end{theorem}

Let us recall also the following abstract result proved in \cite[Thm.3.10]{BB
VMO}, which will be crucial for us:

\begin{theorem}
\label{Thm abstract sing int}Let $\left(  X,d,\mu\right)  $ be a space of
homogeneous type and let $T$ be a singular integral operator which we already
know to be bounded on $L^{p}\left(  X\right)  $ for some $p\in\left(
1,\infty\right)  $. The operator $T$ has kernel $K\left(  x,y\right)  $, which
means that%
\[
Tf\left(  x\right)  =\int_{X}K\left(  x,y\right)  f\left(  y\right)
d\mu\left(  y\right)
\]
when $f$ is compactly supported and $x$ does not belong to
$\operatorname{sprt}f.$ We assume that the kernel $K$ satisfies the mean value
inequality%
\begin{equation}
\left\vert K\left(  x_{0},y\right)  -K\left(  x,y\right)  \right\vert \leq
C\frac{d\left(  x_{0},x\right)  }{d\left(  x_{0},y\right)  B\left(
x;y\right)  }\text{ if }d\left(  x_{0},y\right)  \geq Md\left(  x_{0}%
,x\right)  \label{mean value ineq}%
\end{equation}
where
\[
B\left(  x;y\right)  =\mu\left(  B\left(  x,d\left(  x,y\right)  \right)
\right)
\]
and $M>1$ is such that the condition $d\left(  x_{0},y\right)  \geq Md\left(
x_{0},x\right)  $ implies the equivalence of $d\left(  x_{0},y\right)  $ and
$d\left(  x,y\right)  $ (if $d$ is a distance, like in our case, $M=2$ is a
good choice). Let $\beta>1$ be an exponent such that%
\begin{equation}
\left\vert B\left(  x,kr\right)  \right\vert \leq ck^{\beta}\left\vert
B\left(  x,r\right)  \right\vert \label{growth condition}%
\end{equation}
for every $k\geq1,r>0,x\in X.$

Then, there exists $c>0$ such that for every $f\in L^{p}\left(  X\right)  $,
$x_{0}\in X$, ball $B_{r}=B\left(  \overline{x},r\right)  \ni x_{0}$ (for some
$r>0$, $\overline{x}\in X$), $k\geq M$, we have:%
\begin{align*}
&  \frac{1}{\left\vert B_{r}\left(  \overline{x}\right)  \right\vert }%
\int_{B_{r}}\left\vert Tf\left(  x\right)  -\left(  Tf\right)  _{B_{r}\left(
\overline{x}\right)  }\right\vert d\mu\left(  x\right) \\
&  \leq c\left\{  \frac{1}{k}\mathcal{M}f\left(  x_{0}\right)  +k^{\frac
{\beta}{p}}\left(  \frac{1}{\left\vert B_{kr}\left(  \overline{x}\right)
\right\vert }\int_{B_{kr}\left(  \overline{x}\right)  }\left\vert f\left(
x\right)  \right\vert ^{p}d\mu\left(  x\right)  \right)  ^{1/p}\right\}  .
\end{align*}

\end{theorem}

Finally, we have to define some function spaces of H\"{o}lder or partially
H\"{o}lder continuous functions which will be useful in the following.

\begin{definition}
\label{def:Holderspacesd} Let $f:\mathbb{R}^{N+1}\rightarrow\mathbb{R}$. Given
any number $\alpha\in(0,1)$, we introduce the notation:
\begin{align*}
|f|_{C^{\alpha}(\mathbb{R}^{N+1})}  &  =\sup\left\{  \frac{|f(\xi)-f(\eta
)|}{d(\xi,\eta)^{\alpha}}:\,\text{$\xi,\eta\in$}\mathbb{R}^{N+1}\text{ and
$\xi\neq\eta$}\right\} \\[0.15cm]
\left\vert f\right\vert _{C_{x}^{\alpha}(\mathbb{R}^{N+1})}  &
=\underset{t\in\mathbb{R}}{\operatorname{esssup}}\left\{  \frac
{|f(x,t)-f(y,t)|}{d((x,t),(y,t))^{\alpha}}:\,\text{$x,y\in$ }\mathbb{R}%
^{N},\text{$x\neq y$}\right\} \\
&  =\underset{t\in\mathbb{R}}{\operatorname{esssup}}\left\{  \frac
{|f(x,t)-f(y,t)|}{\Vert x-y\Vert^{\alpha}}:\,\text{$x,y\in$ }\mathbb{R}%
^{N},\text{$x\neq y$}\right\}
\end{align*}
(where the last equality holds by (\ref{d stesso t})). Accordingly, we define
the spaces $C^{\alpha}(\mathbb{R}^{N+1})$ and $C_{x}^{\alpha}(\mathbb{R}%
^{N+1})$ as follows:
\begin{align}
&  C^{\alpha}(\mathbb{R}^{N+1}):=\left\{  f\in C(\mathbb{R}^{N+1})\cap
L^{\infty}(\mathbb{R}^{N+1}):\,|f|_{C^{\alpha}(\mathbb{R}^{N+1})}%
<\infty\right\} \label{eq:defCalfa}\\[0.04in]
&  C_{x}^{\alpha}(\mathbb{R}^{N+1}):=\left\{  f\in L^{\infty}(\mathbb{R}%
^{N+1}):\,|f|_{C_{x}^{\alpha}(\mathbb{R}^{N+1})}<\infty\right\}  .
\label{eq:defCalfax}%
\end{align}
Finally, we define the spaces%
\begin{align}
C_{0}^{\alpha}(\mathbb{R}^{N+1})  &  :=\left\{  f\in C^{\alpha}(\mathbb{R}%
^{N+1}):f\text{ is compactly supported in }\mathbb{R}^{N+1}\right\}
\label{defCalfa0}\\
\mathcal{D}_{x}^{\alpha}(\mathbb{R}^{N+1})  &  :=\left\{  f\in C_{x}^{\alpha
}(\mathbb{R}^{N+1}):f\text{ is compactly supported in }\mathbb{R}%
^{N+1}\right\}  . \label{defDalfax}%
\end{align}

\end{definition}

\section{Operators with measurable coefficients $a_{ij}(t)$%
\label{sec operators t}}

In this Section we consider a KFP operator $\overline{\mathcal{L}}$ of the
form \eqref{L} but with coefficients $a_{ij}$ \emph{only depending on $t$},
that is,
\begin{equation}
\overline{\mathcal{L}}u=\sum_{i,j=1}^{q}a_{ij}(t)\partial_{x_{i}x_{j}}%
^{2}u+Yu,\qquad(x,t)\in\mathbb{R}^{N+1}, \label{eq:LLsolot}%
\end{equation}
satisfying assumptions \textbf{(H1)}\thinspace-\thinspace\textbf{(H2)} in
Section \ref{sec inizio}. For this operator, the goal of this Section is to
prove an estimate on the mean oscillation of $\partial_{x_{i}x_{j}}^{2}u$ over
balls, in terms of $\overline{\mathcal{L}}u$.\ This will be the key result
that we will exploit, in Section \ref{Sec operators a(x,t)}, to get the
desired Sobolev estimates for operators with partially $VMO$ coefficients
$a_{ij}\left(  x,t\right)  $. The result is the following, and will be proved
at the end of Section \ref{subsec oscillation}:

\begin{theorem}
\label{Thm Krylov main step} Let assumptions \textbf{\emph{(H1)}%
}-\textbf{\emph{(H2)}} be in force. Then, for every $p\in(1,\infty)$ there
exists $c>0$ such that, for every $u\in C_{0}^{\infty}(\mathbb{R}^{N+1})$,
$r>0$, $\xi_{0}=(x_{0},t_{0})$ and $\bar{\xi}=(\overline{x},\overline{t}%
)\in\mathbb{R}^{N+1}$ such that $\xi_{0}\in B_{r}(\bar{\xi})$, every $1\leq
i,j\leq q$ and $k\geq4\boldsymbol{\kappa}$ (where $\boldsymbol{\kappa}$ is the
constant in (\ref{quasitriangle})), we have the following estimate:%
\begin{equation}%
\begin{split}
&  \frac{1}{|B_{r}(\bar{\xi})|}\int_{B_{r}(\bar{\xi})}|\partial_{x_{i}x_{j}%
}^{2}u(x,t)-(\partial_{x_{i}x_{j}}^{2}u)_{B_{r}(\bar{\xi})}|\,dxdt\\
&  \qquad\leq c\Big\{\frac{1}{k}\mathcal{M}(\overline{\mathcal{L}}u)(\xi
_{0})\\
&  \qquad\qquad+k^{\frac{Q+2}{p}}\left(  \frac{1}{\left\vert B_{kr}\left(
\overline{\xi}\right)  \right\vert }\int_{B_{kr}\left(  \overline{\xi}\right)
}\left\vert \overline{\mathcal{L}}u(x,t)\right\vert ^{p}\,dx\,dt\right)
^{1/p}\Big\}.
\end{split}
\label{Krylov}%
\end{equation}
The constant $c$ depends on the matrix $A$ only through the number $\nu$ in
\eqref{nu}; moreover, $\mathcal{M}$ is the Hardy-Littlewood maximal operator
defined in \eqref{maximal HL}.
\end{theorem}

\subsection{Fundamental solution and representation formulas for
$\overline{\mathcal{L}}$}

We begin by reviewing some known results concerning the fundamental
so\-lu\-tion of $\overline{\mathcal{L}}$ which we will need in the sequel;
such results are proved in \cite{BB Schauder}, \cite{BP}, to which we refer
for the proofs and for further details. \medskip

To begin with, we state the following \emph{existence} result.

\begin{theorem}
[{See {\cite[Thm.\,3.11]{BB Schauder} and \cite[Thm.\,1.4]{BP}}}%
]\label{Thm fund sol coeff t dip} Under the above assumptions
\emph{\textbf{(H1)}-\textbf{(H2)}}, let $C(t,s)$ be the $N\times N$ matrix
defined as
\begin{equation}
C(t,s)=\int_{s}^{t}E(t-\sigma)\cdot%
\begin{pmatrix}
A_{0}(\sigma) & 0\\
0 & 0
\end{pmatrix}
\cdot E(t-\sigma)^{T}\,d\sigma\quad(\text{with $t>s$}) \label{eq-EC}%
\end{equation}
\emph{(}we recall that $E(\sigma)=\exp(-\sigma B)$, see \eqref{B}\emph{)}.
Then, the matrix $C(t,s)$ is \emph{sym\-me\-tric and positive definite} for
every $t>s$. Moreover, if we define
\begin{equation}%
\begin{split}
&  \Gamma\left(  x,t;y,s\right) \\
&  \quad=\frac{1}{(4\pi)^{N/2}\sqrt{\det C(t,s)}}e^{-\frac{1}{4}\langle
C(t,s)^{-1}(x-E(t-s)y),\,x-E(t-s)y\rangle}\cdot\mathbf{1}_{\{t>s\}}%
\end{split}
\label{eq.exprGammapernoi}%
\end{equation}
\emph{(}where $\mathbf{1}_{A}$ denotes the indicator function of a set
$A$\emph{)}, then for every $u\in C_{0}^{\infty}(\mathbb{R}^{N+1})$ we have
the following representation formula
\begin{equation}
u(x,t)=-\int_{\mathbb{R}^{N+1}}\Gamma(x,t;y,s)\overline{\mathcal{L}%
}u(y,s)\,dy\,ds\quad\text{for every $(x,t)\in\mathbb{R}^{N+1}$},
\label{eq:repru}%
\end{equation}
so that $\Gamma$ is the \emph{fundamental solution} for $\overline
{\mathcal{L}}$ with pole at $(y,s)$. \vspace*{0.1cm}

Moreover, $\Gamma$ satisfies the following properties:

\begin{enumerate}
\item For every fixed $\eta=(y,s)\in\mathbb{R}^{N+1}$, we have
\begin{equation}
(\overline{\mathcal{L}}\Gamma(\cdot;\eta))(x,t)=0\quad\text{for every
$x\in\mathbb{R}^{N}$ and a.e.\thinspace$t\in\mathbb{R}$}.
\label{eq:LLGammaZero}%
\end{equation}

\item For every fixed $x\in\mathbb{R}^{N}$ and every $t>s$, we have
\begin{equation}
\label{eq:integralGamma1}\int_{\mathbb{R}^{N}}\Gamma(x,t;y,s)\,dy=1.
\end{equation}

\end{enumerate}
\end{theorem}

\begin{remark}
\label{rem:MorePropGamma} It is worth mentioning that the existence and the
explicit expression of a fundamental solution $\Gamma$ for the operator
$\overline{\mathcal{L}}$ is proved in \cite[Thm.\,1.4]{BP} under a
\emph{weaker ver\-sion} of assumption \textbf{(H2)}; moreover, several other
properties of $\Gamma$ (which we will not exploit in this paper) are established.
\end{remark}

In the particular case when the coefficients $a_{ij}$ of $\overline
{\mathcal{L}}$ are \emph{constant}, the associated fundamental solution
$\Gamma$ constructed in Theorem \ref{Thm fund sol coeff t dip} takes a simpler
form; due to its relevance in our arguments (see Theorem
\ref{thm:finepropGamma} and the proof of Lemma \ref{lem:PropKe}), we
explicitly state this expression in the next theorem.

\begin{theorem}
[See \cite{BP}]\label{Thm fund sol cost coeff} Let $\alpha>0$ be fixed, and
let
\begin{equation}
\mathcal{L}_{\alpha}u=\alpha\sum_{i=1}^{q}\partial_{x_{i}x_{i}}^{2}u+Yu.
\label{L-alpha}%
\end{equation}
Moreover, let $\Gamma_{\alpha}$ be the fundamental solution of $\mathcal{L}%
_{\alpha}$, whose existence is guaranteed by Theorem
\ref{Thm fund sol coeff t dip}. Then, the following facts hold true:

\begin{enumerate}
\item $\Gamma_{\alpha}$ is a \emph{kernel of convolution type}, that is,
\begin{equation}
\label{eq:Gammaalfaconvolution}%
\begin{split}
\Gamma_{\alpha}(x,t;y,s)  &  =\Gamma_{\alpha}\big(x-E(t-s)y,t-s;0,0\big)\\
&  = \Gamma_{\alpha}\big((y,s)^{-1}\circ(x,t);0,0\big);
\end{split}
\end{equation}

\item the matrix $C(t,s)$ in \eqref{eq-EC} takes the simpler form
\begin{equation}
C(t,s)=C_{0}(t-s), \label{C_0}%
\end{equation}
where $C_{0}(\tau)$ is the $N\times N$ matrix defined as
\[
C_{0}(\tau)=\alpha\int_{0}^{\tau}E(t-\sigma)\cdot%
\begin{pmatrix}
I_{q} & 0\\
0 & 0
\end{pmatrix}
\cdot E(t-\sigma)^{T}d\sigma\qquad\forall\,\,\tau> 0.
\]
Furthermore, one has the `homogeneity property'
\begin{equation}
C_{0}(\tau)=D_{0}(\sqrt{\tau})C_{0}(1)D_{0}(\sqrt{\tau})\qquad\forall
\,\,\tau>0. \label{C omogenea}%
\end{equation}

\end{enumerate}

In particular, by combining \eqref{eq.exprGammapernoi} with
\eqref{C_0}-\eqref{C omogenea}, we can write%
\begin{equation}%
\begin{split}
&  \Gamma_{\alpha}(x,t;0,0)=\frac{1}{(4\pi\alpha)^{N/2}\sqrt{\det C_{0}(t)}%
}e^{-\frac{1}{4\alpha}\left\langle C_{0}(t)^{-1}x,x\right\rangle }\\
&  \qquad=\frac{1}{(4\pi\alpha)^{N/2}t^{Q/2}\sqrt{\det C_{0}(1)}}e^{-\frac
{1}{4\alpha}\langle C_{0}(1)^{-1}\big(D_{0}\big(\frac{1}{\sqrt{t}%
}\big)x\big),\,D_{0}\big(\frac{1}{\sqrt{t}}\big)x\rangle}.
\end{split}
\label{eq.exprGammaalfa}%
\end{equation}

\end{theorem}

With Theorems \ref{Thm fund sol coeff t dip}-\ref{Thm fund sol cost coeff} at
hand, we then proceed by recalling some \emph{fine properties} of $\Gamma$ and
of its (spatial) derivatives, which will play a key r\^{o}le in our argument.
To clearly state such properties (see Theorem \ref{thm:finepropGamma} below),
we first introduce a notation: if $\boldsymbol{\ell}=(\ell_{1},\ldots,\ell
_{N})\in(\N\cup\{0\})^{N}$ is a given multi-index, we set
\[
D_{x}^{\boldsymbol{\ell}}=\partial_{x_{1}}^{\ell_{1}}\cdots\partial_{x_{N}%
}^{\ell_{N}},\qquad\omega(\boldsymbol{\ell})=\textstyle\sum_{i=1}^{N}\ell
_{i}q_{i},
\]
where $q_{1},\ldots,q_{N}$ are the exponents appearing in the dilation
$D_{0}(\lambda)$, see \eqref{dilations}.

\begin{theorem}
[Fine properties of $\Gamma$]\label{thm:finepropGamma} (See
\cite[Thm.\thinspace3.5, Thm.\thinspace3.9, Prop.\,3.13 and Thm.\thinspace
3.16]{BB Schauder}) Let $\Gamma$ be as in Theorem
\ref{Thm fund sol coeff t dip}, and let $\nu>0$ be as in \eqref{nu}. Then, the
following assertions hold.

\begin{enumerate}
\item There exists $c_{1}>0$ and, for every pair of multi-in\-dices
$\bd\ell_{1},\bd\ell_{2}\in(\N\cup\{0\})^{N}$, there exists $c=c(\nu
,\bd\ell_{1},\bd\ell_{2})>0$, such that, for every $(x,t),(y,s)\in
\mathbb{R}^{N+1}$ with $t\neq s$, we have
\begin{equation}%
\begin{split}
|D_{x}^{\bd\ell_{1}}D_{y}^{\bd\ell_{2}}\Gamma(x,t;y,s)|  &  \leq\frac{{c}%
}{(t-s)^{\omega(\bd\ell_{1}+\bd\ell_{2})/2}}\,\Gamma_{c_{1}\nu^{-1}%
}(x,t;y,s)\\
&  \leq\frac{c}{d((x,t),(y,s))^{Q+\omega(\bd\ell_{1}+\bd\ell_{2})}},
\end{split}
\label{eq:mainestimGaussian}%
\end{equation}
In particular, we have
\begin{equation}
|D_{x}^{\bd\ell_{1}}D_{y}^{\bd\ell_{2}}\Gamma(\xi,\eta)|\leq\frac{c}%
{d(\xi,\eta)^{Q+\omega(\bd\ell_{1}+\bd\ell_{2})}}\quad\forall\,\,\xi\neq\eta.
\label{eq:estimDerGammaStandard}%
\end{equation}

\item For every multi-index $\bd\ell\in(\N\cup\{0\})^{N}$ there exists
$c=c(\bd\ell,\nu)>0$ such that
\begin{align}
&  |D_{x}^{\bd{\ell}}\Gamma(\xi_{1},\eta)-D_{x}^{\bd{\ell}}\Gamma(\xi_{2}%
,\eta)|\leq c\frac{d(\xi_{1},\xi_{2})}{d(\xi_{1},\eta)^{Q+\omega(\bd\alpha
)+1}}\label{eq:meanvalueGamma}\\
&  |D_{x}^{\bd{\ell}}\Gamma(\eta,\xi_{1})-D_{x}^{\bd{\ell}}\Gamma(\eta,\xi
_{2})|\leq c\frac{d(\xi_{1},\xi_{2})}{d(\xi_{1},\eta)^{Q+\omega(\bd\alpha)+1}}
\label{eq:meanvalueGammaScambiate}%
\end{align}
for every $\xi_{1},\xi_{2},\eta\in\mathbb{R}^{N+1}$ such that
\[
d(\xi_{1},\eta)\geq4\bd{\kappa}d(\xi_{1},\xi_{2})>0.
\]

\item Let $\alpha\in(0,1)$ be fixed, and let $1\leq i,j\leq q$. Then, there
exists a constant $c = c(\alpha) > 0$ such that, for every $x\in\mathbb{R}%
^{N}$ and every $\tau< t$, one has
\begin{equation}
\label{eq:ExProp313dacitare}\int_{\mathbb{R}^{N}\times(\tau,t)}|\partial
^{2}_{x_{i}x_{j}}\Gamma(x,t;y,s)|\cdot\|E(s-t)x-y\|^{\alpha}\,dy\,ds \leq
c(t-\tau)^{\alpha/2}.
\end{equation}

\item There exists a con\-stant $c>0$ such that, for every fixed $1\leq
i,j\leq q$, one has the estimate
\begin{align}
I_{r,\tau}(x,t)  &  :=\int_{\tau}^{t}\left\vert \int_{\{y\in\mathbb{R}%
^{N}:\,d((x,t),(y,s))\,\geq\,r\}}\partial_{x_{i}x_{j}}^{2}\Gamma
(x,t;y,s)\,dy\,\right\vert \,ds\leq c,\label{eq:cancelprop}\\
J_{r,\tau}(y,s)  &  :=\int_{s}^{\tau}\left\vert \int_{\{x\in\mathbb{R}%
^{N}:\,d((y,s),(x,t))\,\geq\,r\}}\partial_{x_{i}x_{j}}^{2}\Gamma
(x,t;y,s)\,dx\,\right\vert \,dt\leq c, \label{eq:cancelpropScambiate}%
\end{align}
for every $(x,t),(y,s)\in\mathbb{R}^{N+1},s<\tau<t,r>0$.
\end{enumerate}
\end{theorem}

\begin{remark}
\label{rem:MeanValueScambiate} As a matter of fact, the mean-value inequality
\eqref{eq:meanvalueGammaScambiate} is not explicitly proved in \cite{BB
Schauder}; however, the proof of this inequality is \emph{totally analogous}
to that of \eqref{eq:meanvalueGamma} (see, precisely, \cite[Thm.\,3.9]{BB
Schauder}), and it relies on the `subelliptic' mean value theorem for the
system of H\"{o}rmander vector fields
\[
\{\partial_{x_{1}},\ldots,\partial_{x_{q}},Y\}
\]
(see \cite[Thm.\,2.1]{BB Schauder}), together with the estimates
\eqref{eq:estimDerGammaStandard} (which apply to \emph{every spatial
derivative} of $\Gamma$, both with respect to $x$ and $y$). The same comment
applies to the cancellation property \eqref{eq:cancelpropScambiate} (see the
proof of \cite[Thm.\,3.16]{BB Schauder}).
\end{remark}

Before proceeding we highlight, for a future reference, an easy consequence of
Theorem \ref{thm:finepropGamma} and of i\-dentity \eqref{eq:integralGamma1},
which will be repeatedly used in the sequel.

\begin{lemma}
\label{lem:integralvanishinggamma} Let $\Gamma$ be as in Theorem
\ref{Thm fund sol coeff t dip}, and let $\boldsymbol{\ell}$ be a fixed
non-zero multi-index. Then, for every $x\in\mathbb{R}^{N}$ and every $s<t$ we
have
\begin{equation}
\int_{\mathbb{R}^{N}}D_{x}^{\boldsymbol{\ell}}\Gamma(x,t;y,s)\,dy=0.
\label{eq:intDxGammazero}%
\end{equation}

\end{lemma}

Starting from the representation formula \eqref{eq:repru}, and using the fine
properties of $\Gamma$ collected in Theorem \ref{thm:finepropGamma}, one can
prove the following representation formula for the
second-order derivatives of a function $u\in C_{0}^{\infty}(\mathbb{R}^{N+1}%
)$. Throughout what follows, we tacitly understand that $\alpha$ is a fixed
number in $(0,1)$.

\begin{theorem}
[{See \cite[Cor.\,3.12 and Thm.\,3.14]{BB Schauder}}]%
\label{Thm repr formula u} For every $u\in C_{0}^{\infty}(\mathbb{R}^{N+1})$
and for every $1\leq i,j\leq q$, we have the following representation
formulas
\begin{align}
\partial_{x_{i}}u(x,t)  &  =-\int_{\mathbb{R}^{N+1}}\partial_{x_{i}}%
\Gamma(x,t;\cdot)\,\overline{\mathcal{L}}u\,dy\,ds,\label{eq:reprudexi}\\
\partial_{x_{i}x_{j}}^{2}u(x,t)  &  =\int_{\mathbb{R}^{N+1}}\partial
_{x_{i}x_{j}}^{2}\Gamma(x,t;y,s)\left[  \overline{\mathcal{L}}%
u(E(s-t)x,s)-\overline{\mathcal{L}}u(y,s)\right]  \,dy\,ds,
\label{repr formula u_xx}%
\end{align}
holding true for every $(x,t)\in\mathbb{R}^{N+1}$. In particular, the two
integrals appearing in the above \eqref{eq:reprudexi}\thinspace-\thinspace
\eqref{repr formula u_xx} are absolutely convergent, and the operator
\begin{equation}
T_{ij}f(x,t)=\int_{\mathbb{R}^{N+1}}\partial_{x_{i}x_{j}}^{2}\Gamma
(x,t;y,s)\big[f(E(s-t)x,s)-f(y,s)\big]\,dy\,ds \label{T}%
\end{equation}
is well defined for every $f\in\mathcal{D}_{x}^{\alpha}(\mathbb{R}^{N+1})$
(this space has been defined in (\ref{defDalfax})).
\end{theorem}

\subsection{Global $W^{2,p}$ and mean-oscillation estimates for $\overline
{\mathcal{L}}$\label{subsec oscillation}}

Taking into account the results recalled so far and following the strategy
described at the beginning of the Section, we now turn to establish
\emph{global $W^{2,p}$ and mean-oscillation estimates for the solutions of
$\overline{\mathcal{L}}u=f$}. Our first result in this direction is the following.

\begin{theorem}
\label{thm:LpEstimate}Let assumptions \textbf{\emph{(H1)}}\thinspace
-\thinspace\textbf{\emph{(H2)}} be in force, and let $p\in\left(
1,\infty\right)  $ be fixed. Moreover, let $T_{ij}$ be the operator defined in
\eqref{T}, with $1\leq i,j\leq q$.

Then, $T_{ij}$ can be extended to a linear and continuous operator from
$L^{p}(\mathbb{R}^{N+1})$ into itself. In particular, there exists $c=c(p)>0$
such that
\begin{equation}
\sum_{i,j=1}^{q}\Vert\partial_{x_{i}x_{j}}^{2}u\Vert_{L^{p}(\mathbb{R}^{N+1}%
)}\leq c\Vert\overline{\mathcal{L}}u\Vert_{L^{p}(\mathbb{R}^{N+1})},
\label{Lp Lbar}%
\end{equation}
for every $u\in C_{0}^{\infty}(\mathbb{R}^{N+1})$.
\end{theorem}

\begin{remark}
We wish to stress that, in view of our general strategy, the relevant point of
the previous theorem is the $L^{p}$ continutiy of the operator $T_{ij}$ and
not the estimate (\ref{Lp Lbar}) \emph{per se}. Also, we recall that, in the
paper \cite{M}, $L^{p}$ estimates for $\partial_{x_{i}x_{j}}^{2}u$ have been
proved for KFP operators with coefficients $a_{ij}\left(  x,t\right)  $
uniformly continuous in $x$ and $L^{\infty}$ in $t$. Since our operator
$\overline{\mathcal{L}}$ with coefficients only depending on $t$ can be seen
as a special case of the operators studied in \cite{M}, the estimate
(\ref{Lp Lbar}) should be contained in the results proved in \cite{M}.
However, we could not find in that paper neither a statement nor a proof of an
explicit representation formula of the kind $\partial_{x_{i}x_{j}}^{2}%
u=T_{ij}\left(  \overline{\mathcal{L}}u\right)  $, linking the derivatives
$\partial_{x_{i}x_{j}}^{2}u$ to a specific singular integral operator. The
arguments in this Section make our proof logically independent from the
results in \cite{M}, which we cannot apply directly to our context.
\end{remark}

In order to prove Theorem \ref{thm:LpEstimate}, we first observe that, if
$u\in C_{0}^{\infty}(\mathbb{R}^{N+1})$, then $\overline{\mathcal{L}}%
u\in\mathcal{D}_{x}^{\alpha}(\mathbb{R}^{N+1})$. Hence, if $1\leq i,j\leq q$
are fixed, by the representation formula \eqref{repr formula u_xx} we can
write%
\begin{align*}
\partial_{x_{i}x_{j}}^{2}u(x,t)  &  =\int_{\mathbb{R}^{N+1}}\partial
_{x_{i}x_{j}}^{2}\Gamma(x,t;y,s)\left[  \overline{\mathcal{L}}%
u(E(s-t)x,s)-\overline{\mathcal{L}}u(y,s)\right]  dyds\\
&  =T_{ij}(\overline{\mathcal{L}}u),
\end{align*}
where $T_{ij}$ is the operator defined in \eqref{T}. On the other hand, since
by Theorem \ref{Thm repr formula u} the above integral \emph{converges
absolutely}, we have
\begin{align}
T_{ij}(\overline{\mathcal{L}}u)  &  =\lim_{\varepsilon\rightarrow0^{+}}%
\int_{-\infty}^{t-\varepsilon}\int_{\mathbb{R}^{N}}\partial_{x_{i}x_{j}}%
^{2}\Gamma(x,t;y,s)\left[  \overline{\mathcal{L}}u(E(s-t)x,s)-\overline
{\mathcal{L}}u(y,s)\right]  dyds\nonumber\\
&  =\lim_{\varepsilon\rightarrow0^{+}}\Big(-\int_{-\infty}^{t-\varepsilon}%
\int_{\mathbb{R}^{N}}\partial_{x_{i}x_{j}}^{2}\Gamma(\xi;\eta)\overline
{\mathcal{L}}u(\eta)\,d\eta\Big), \label{limit sing int1}%
\end{align}
where we have used the fact that
\begin{align*}
&  \int_{-\infty}^{t-\varepsilon}\int_{\mathbb{R}^{N}}\Gamma(x,t;y,s)\overline
{\mathcal{L}}u(E(s-t)x,s)\,dyds\\
&  \qquad=\int_{-\infty}^{t-\varepsilon}\overline{\mathcal{L}}%
u(E(s-t)x,s)\Big(\int_{\mathbb{R}^{N}}\partial_{x_{i}x_{j}}^{2}\Gamma
(x,t;y,s)\,dy\Big)ds=0,
\end{align*}
see Lemma \ref{lem:integralvanishinggamma} (and take into account the
regularity of $\Gamma$ out of the pole).

In order to apply $L^{p}$ continuity results of singular integrals, we need to
rewrite the limit (\ref{limit sing int1}) in a different way. To this aim, let
us introduce a cut-off function $\phi_{\varepsilon}\in C_{0}^{\infty
}(\mathbb{R})$ such that

\begin{itemize}
\item[i)] $0\leq\phi_{\varepsilon}\leq1$ on $\mathbb{R}$;

\item[ii)] $\phi_{\varepsilon}=0$ on $(-\infty,\varepsilon]$ and
$\phi_{\varepsilon}=1$ on $[2\varepsilon,+\infty)$;

\item[iii)] $\left\vert \phi_{\varepsilon}^{\prime}\right\vert \leq
c/\varepsilon$ (for some $c>0$),
\end{itemize}

\noindent and define the \emph{truncated operator}
\begin{equation}
T_{ij}^{\varepsilon}(f)(x,t)=-\int_{\mathbb{R}^{N+1}}\phi_{\varepsilon
}(t-s)\partial_{x_{i}x_{j}}^{2}\Gamma(x,t;y,s)f(y,s)\,dy\,ds\text{,}
\label{eq:opTepsilon}%
\end{equation}
which is well defined for every compactly supported $f\in L^{\infty
}(\mathbb{R}^{N+1})$ since for every $\left(  x,t\right)  $, the function%
\[
\left(  y,s\right)  \mapsto\phi_{\varepsilon}(t-s)\partial_{x_{i}x_{j}}%
^{2}\Gamma(x,t;y,s)
\]
is locally integrable (see Lemma \ref{lem:PropKe} below). We then claim that,
for every fixed $f\in\mathcal{D}_{x}^{\alpha}(\mathbb{R}^{N+1})$ and for every
$\xi\in\mathbb{R}^{N+1}$, we have
\begin{equation}
T_{ij}(f)(\xi)=\lim_{\varepsilon\rightarrow0^{+}}T_{ij}^{\varepsilon}(f)(\xi).
\label{eq:TelimitT}%
\end{equation}
Indeed, by the above computation we can write, for $f\in\mathcal{D}%
_{x}^{\alpha}(\mathbb{R}^{N+1})$,%
\begin{align*}
\mathrm{a)}\,\,  &  T_{ij}\left(  f\right)  \left(  x,t\right)  =\int%
_{\mathbb{R}^{N+1}}\partial_{x_{i}x_{j}}^{2}\Gamma(x,t;y,s)\left[
f(E(s-t)x,s)-f(y,s)\right]  dyds;\\
\mathrm{b)}\,\,  &  T_{ij}^{\varepsilon}\left(  f\right)  (x,t) =-\int%
_{\mathbb{R}^{N+1}}\phi_{\varepsilon}(t-s)\partial_{x_{i}x_{j}}^{2}%
\Gamma(x,t;y,s)f(y,s)\,dy\,ds\\
&  \qquad=\int_{\mathbb{R}^{N+1}}\phi_{\varepsilon}(t-s)\partial_{x_{i}x_{j}%
}^{2}\Gamma(x,t;y,s)\left[  f(E(s-t)x,s)-f(y,s)\right]  \,dy\,ds;
\end{align*}
hence, by using \eqref{eq:ExProp313dacitare} in Theorem
\ref{thm:finepropGamma} we have
\begin{align*}
&  \left\vert \left(  T_{ij}\left(  f\right)  -T_{ij}^{\varepsilon}\left(
f\right)  \right)  (x,t)\right\vert \\
&  =\left\vert \int_{t-2\varepsilon}^{t}\left[  1-\phi_{\varepsilon
}(t-s)\right]  \left(  \int_{\mathbb{R}^{N}}\partial_{x_{i}x_{j}}^{2}%
\Gamma(x,t;y,s)\left[  f(E(s-t)x,s)-f(y,s)\right]  \,dy\right)
\,ds\right\vert \\
&  \leq\left\vert f\right\vert _{C_{x}^{\alpha}(\mathbb{R}^{N+1})}%
\int_{t-2\varepsilon}^{t}\left(  \int_{\mathbb{R}^{N}}\left\vert
\partial_{x_{i}x_{j}}^{2}\Gamma(x,t;y,s)\left\Vert E(s-t)x-y\right\Vert
^{\alpha}\,\right\vert dy\,\right)  ds\\
&  \leq c\left\vert f\right\vert _{C_{x}^{\alpha}(\mathbb{R}^{N+1}%
)}\varepsilon^{\alpha/2}.
\end{align*}
This proves \eqref{eq:TelimitT}, and allows us to rewrite
\eqref{repr formula u_xx} as follows:
\begin{equation}%
\begin{split}
\partial_{x_{i}x_{j}}^{2}u(x,t)  &  =\lim_{\varepsilon\rightarrow0^{+}%
}\Big(-\int_{\mathbb{R}^{N+1}}\phi_{\varepsilon}(t-s)\partial_{x_{i}x_{j}}%
^{2}\Gamma(x,t;y,s)\overline{\mathcal{L}}u(y,s)\,dy\,ds\Big)\\
&  \equiv\lim_{\varepsilon\rightarrow0^{+}}T_{ij}^{\varepsilon}(\overline
{\mathcal{L}}u),
\end{split}
\label{eq:reprLimite}%
\end{equation}

\bigskip

In view of the above facts, in order to prove Theorem \ref{thm:LpEstimate} it
then suffices to show the following result.

\begin{proposition}
\label{prop:ContTepsilon} Let $p\in\left(  1,\infty\right)  $ be fixed, and
let $T_{ij}^{\varepsilon}$ be the operator defined in \eqref{eq:opTepsilon}.
Then, $T_{ij}^{\varepsilon}$ can be extended to a linear and continuous
operator from $L^{p}(\mathbb{R}^{N+1})$ into itself: there exists
$c=c(\nu,p)>0$ such that
\begin{equation}
\left\Vert T_{ij}^{\varepsilon}f\right\Vert _{L^{p}\left(  \mathbb{R}%
^{N+1}\right)  }\leq c\left\Vert f\right\Vert _{L^{p}\left(  \mathbb{R}%
^{N+1}\right)  }\text{ \ \ for every }f\in L^{p}\left(  \mathbb{R}%
^{N+1}\right)  . \label{eq:LpEstimTe}%
\end{equation}
We stress that the constant $c$ appearing in \eqref{eq:LpEstimTe} is
independent of $\varepsilon$.
\end{proposition}

Before embarking on the (quite technical) proof of Proposition
\ref{prop:ContTepsilon}, let us show how this proposition allows us to easily
prove Theorem \ref{thm:LpEstimate}.

\begin{proof}
[Proof (of Theorem \ref{thm:LpEstimate})]We first observe that, given any
$f\in\mathcal{D}_{x}^{\alpha}(\mathbb{R}^{N+1})$, by combi\-ning
\eqref{eq:LpEstimTe} with Fatou's lemma, we get
\begin{align*}
\int_{\mathbb{R}^{N+1}}|T_{ij}(f)|^{p}\,d\xi &  =\int_{\mathbb{R}^{N+1}}%
\lim_{\varepsilon\rightarrow0^{+}}|T_{ij}^{\varepsilon}(f)|^{p}\,d\xi\\
&  \leq\liminf_{\varepsilon\rightarrow0^{+}}\int_{\mathbb{R}^{N+1}}%
|T_{ij}^{\varepsilon}(f)|^{p}\,d\xi\leq c^{p}\Vert f\Vert_{L^{p}%
(\mathbb{R}^{N+1})}^{p}.
\end{align*}
From this, since $\mathcal{D}_{x}^{\alpha}(\mathbb{R}^{N+1})$ is dense in
$L^{p}(\mathbb{R}^{N+1})$, we conclude that $T_{ij}$ can be e\-xtended to a
linear and continuous operator from $L^{p}(\mathbb{R}^{N+1})$ into itself.

In particular, given any $u\in C_{0}^{\infty}(\mathbb{R}^{N+1})$, since
$f=\overline{\mathcal{L}}u\in\mathcal{D}_{x}^{\alpha}(\mathbb{R}^{N+1})$, from
the representation formula \eqref{repr formula u_xx} we get
\[
\Vert\partial_{x_{i}x_{j}}^{2}u\Vert_{L^{p}(\mathbb{R}^{N+1})}=\Vert
T_{ij}(\overline{\mathcal{L}}u)\Vert_{L^{p}(\mathbb{R}^{N+1})}\leq
c\Vert\overline{\mathcal{L}}u\Vert_{L^{p}(\mathbb{R}^{N+1})}.
\]
This ends the proof.
\end{proof}

We now turn to prove Proposition \ref{prop:ContTepsilon}. We rely on the
following abstract result from the theory of singular integrals.

\begin{theorem}
[{See \cite[Thm.\,4.1]{BCsing}}]\label{Thm spazio omogeneo} Let $K:\mathbb{R}%
^{N+1}\times\mathbb{R}^{N+1}\setminus\left\{  \xi\neq\eta\right\}
\rightarrow\mathbb{R}$ be a kernel satisfying the following conditions: there
exist $\boldsymbol{A},\boldsymbol{B},\boldsymbol{C}>0$ such that:

\begin{enumerate}
\item[(i)] for every $\xi\neq\eta\in\mathbb{R}^{N+1}$ one has
\[
\left\vert K(\xi,\eta)\right\vert +\left\vert K(\eta,\xi)\right\vert \leq
\frac{\boldsymbol{A}}{\left\vert B_{d(\xi,\eta)}(\xi)\right\vert };
\]

\item[(ii)] there exists a constant $\beta>0$ such that for every $\xi,\xi
_{0},\eta\in\mathbb{R}^{N+1}$ satisfying the condition $d(\xi_{0},\eta
)\geq4\bd\kappa\,d(\xi_{0},\xi)>0$, we have
\begin{align*}
&  \left\vert K(\xi,\eta)-K(\xi_{0},\eta)\right\vert +\left\vert K(\eta
,\xi)-K(\eta,\xi_{0})\right\vert \\
&  \qquad\leq\boldsymbol{B}\,\left(  \frac{d(\xi_{0},\xi)}{d(\xi_{0},\eta
)}\right)  ^{\beta}\cdot\frac{1}{\left\vert B_{d(\xi_{0},\eta)}(\xi
_{0})\right\vert };
\end{align*}

\item[(iii)] for every $\zeta\in\mathbb{R}^{N+1}$ and every $0<r<R<\infty$,
one has
\[
\left\vert \int_{\{r<d(\zeta,\eta)<R\}}K(\zeta,\eta)\,d\eta\right\vert
+\left\vert \int_{\{r<d(\zeta,\xi)<R\}}K(\xi,\zeta)\,d\xi\right\vert
\leq\boldsymbol{C}.
\]

\end{enumerate}

Suppose, in addition, that $K_{\varepsilon}$ \emph{(}for $\varepsilon
>0$\emph{)} is a `regularized kernel' defined in such a way that the following
properties holds:

\begin{itemize}
\item[(a)] $K_{\varepsilon}(\xi,\cdot)$ and $K_{\varepsilon}(\cdot,\xi)$ are
locally integrable for every $\xi\in\mathbb{R}^{N+1}$;

\item[(b)] $K_{\varepsilon}$ satisfies the standard estimates (i), (ii), with
constant bounded by $c^{\prime}\left(  \boldsymbol{A}+\boldsymbol{B}\right)  $
and $c^{\prime}$ absolute constant (independent of $\varepsilon$);

\item[(c)] there exists an absolute constant $c^{\prime}$ such that for every
$\zeta\in\mathbb{R}^{N+1}$%
\[
\left\vert \int_{\{r<d(\zeta,\eta)<R\}}K_{\varepsilon}(\zeta,\eta
)\,d\eta\right\vert +\left\vert \int_{\{r<d(\zeta,\xi)<R\}}K_{\varepsilon}%
(\xi,\zeta)\,d\xi\right\vert \leq c^{\prime}\boldsymbol{C}.
\]

\end{itemize}

For every $f\in C_{0}^{\alpha}(\mathbb{R}^{N+1})$ (see (\ref{defCalfa0})) set%
\[
T_{\varepsilon}f(\xi)=\int_{\mathbb{R}^{N+1}}K_{\varepsilon}(\xi,\eta
)f(\eta)\,d\eta.
\]
Then, for every $p\in(1,\infty)$, the operator $T_{\varepsilon}$ can be
extended to a linear continuous operator on $L^{p}(\mathbb{R}^{N+1})$, and%
\[
\left\Vert T_{\varepsilon}f\right\Vert _{L^{p}(\mathbb{R}^{N+1})}\leq
c\left\Vert f\right\Vert _{L^{p}(\mathbb{R}^{N+1})}%
\]
for every $f\in L^{p}(\mathbb{R}^{N+1})$, with $c>0$ independent of
$\varepsilon.$ Moreover, the constant $c$ in the last estimate depends on the
quantities involved in the assumptions as follows:%
\[
c\leq c^{\prime}\left(  \boldsymbol{A}+\boldsymbol{B}+\boldsymbol{C}\right)
,
\]
with $c^{\prime}$ `absolute' constant.
\end{theorem}

\begin{remark}
\label{rem:TheoremMoreGeneral} It is worth mentioning that Theorem
\ref{Thm spazio omogeneo} can be applied in the present context since
$(\mathbb{R}^{N+1},d,\left\vert \cdot\right\vert )$ is a \emph{space of
homogeneous type}, in the sense of Coifman-Weiss (see Remark \ref{rem:propd}).
In fact, such a result actually holds if we replace the space $(\mathbb{R}%
^{N+1},d,\left\vert \cdot\right\vert )$ with any \emph{space of homogeneous
type} $(X,d,\mu)$. More precisely, the standard definition of space of
homogeneous type requires $d\left(  x,y\right)  $ to be symmetric, while our
$d$ is actually only quasisymmetric (see (\ref{quasisymmetric})). However,
this problem can be overcome in a standard way: given a quasisymmetric
quasidistance $d$, the function
\[
d^{\prime}\left(  x,y\right)  =d\left(  x,y\right)  +d\left(  y,x\right)
\]
is a (symmetric) quasidistance, equivalent to $d$. Now, properties (i)-(ii) in
Theorem \ref{Thm spazio omogeneo} are clearly stable under replacement of $d$
with an equivalent $d^{\prime}$. This is also true for the cancellation
property (iii), for a less obvious reason, which is discussed for instance in
\cite[Remark 4.6]{BCsing}.
\end{remark}

In order to deduce Proposition \ref{prop:ContTepsilon} from Theorem
\ref{Thm spazio omogeneo} we first note that, if $T_{ij}^{\varepsilon}$ is the
operator defined in \eqref{eq:opTepsilon}, we can clearly write
\[
T_{ij}^{\varepsilon}(f)=\int_{\mathbb{R}^{N+1}}K_{ij}^{\varepsilon}(\xi
;\eta)f(\eta)\,d\eta,
\]
where the kernel $K_{ij}^{\varepsilon}$ is given by
\begin{equation}
K_{ij}^{\varepsilon}(\xi;\eta)=-\phi_{\varepsilon}(t-s)\partial_{x_{i}x_{j}%
}^{2}\Gamma(\xi;\eta). \label{eq:defKe}%
\end{equation}
Hence, if we set $K_{ij}(\xi;\eta)=-\partial_{x_{i}x_{j}}^{2}\Gamma(\xi;\eta
)$, we can directly derive Proposition \ref{prop:ContTepsilon} from Theorem
\ref{Thm spazio omogeneo} as soon as we are able to prove the following facts:

\begin{itemize}
\item[1)] $K_{ij}$ satisfies properties (i)-\thinspace(iii) in the statement
of Theorem \ref{Thm spazio omogeneo};

\item[2)] $K_{ij}^{\varepsilon}$ satisfies properties (a)-(c) in the statement
of Theorem \ref{Thm spazio omogeneo}.
\end{itemize}

Since the validity of assertion 1) follows directly from Theorem
\ref{thm:finepropGamma} (jointly with the fact that $|B_{r}(\xi)| =
c\,r^{Q+2}$, see \eqref{measure ball}), we turn to prove assertion 2).

\begin{lemma}
\label{lem:PropKe} The `regularized kernel' $K_{ij}^{\varepsilon}$ defined in
\eqref{eq:defKe} satisfies the proper\-ties \emph{(a)}\thinspace
-\thinspace\emph{(c)} in the statement of Theorem \ref{Thm spazio omogeneo}.
\end{lemma}

\begin{proof}
We prove the validity of the three properties separately. \medskip

\noindent-\thinspace\thinspace\emph{Proof of property} (a). On account of
Theorem \ref{thm:finepropGamma}\thinspace-\thinspace1), and recalling all the
pro\-perties satisfied by $\phi_{\varepsilon}$, for every fixed $\xi
=(x,t)\in\mathbb{R}^{N+1}$ we have
\begin{align*}
|K_{ij}^{\varepsilon}(\xi,\cdot)|  &  =|\phi_{\varepsilon}(t-s)\partial
_{x_{i}x_{j}}^{2}\Gamma(\xi;\eta)|\leq c\,\mathbf{1}_{\{t-s\geq\varepsilon
\}}(s)\,\frac{1}{d(\xi,\eta)^{Q}}\\
&  (\text{by the explicit expression of $d$, see \eqref{d}})\\
&  \leq c\,\mathbf{1}_{\{t-s\geq\varepsilon\}}(s)\,\frac{1}{|t-s|^{Q/2}}%
\leq\frac{c}{\varepsilon^{Q/2}},
\end{align*}
and this proves that $K_{ij}^{\varepsilon}(\xi,\cdot)\in L_{\mathrm{loc}}%
^{1}(\mathbb{R}^{N+1})$. Analogously one can show that $K_{ij}^{\varepsilon
}(\cdot,\xi)$ is bounded, and thus locally integrable. \medskip

\noindent-\thinspace\thinspace\emph{Proof of property} (b). We begin by
proving that $K_{ij}^{\varepsilon}$ satisfies the standard estima\-te (i)
(with a constant $\leq c^{\prime}(\mathbf{A}+\mathbf{B})$ and $c^{\prime}$
independent of $\varepsilon$). To this end it suff\-ices to observe that,
since $0\leq\phi_{\varepsilon}\leq1$, for every $\xi,\eta\in\mathbb{R}^{N+1}$
we have
\begin{align*}
|K_{ij}^{\varepsilon}(\xi,\eta)|+|K_{ij}^{\varepsilon}(\eta,\xi)|  &
\leq|\partial_{x_{i}x_{j}}^{2}\Gamma(\xi,\eta)|+|\partial_{x_{i}x_{j}}%
^{2}\Gamma(\eta,\xi)|\\
&  \leq\frac{\boldsymbol{A}}{d(\xi,\eta)^{Q+2}} = \frac{\boldsymbol{A}%
}{|B_{d(\xi,\eta)}(\xi)|},
\end{align*}
where we have used the fact that $K_{ij}=-\partial_{x_{i}x_{j}}^{2}\Gamma$
satisfies (i), see \eqref{eq:estimDerGammaStandard}.

We then turn to prove the validity of the standard estimate (b), namely
\begin{equation}%
\begin{split}
&  |K_{ij}^{\varepsilon}(\xi,\eta)-K_{ij}^{\varepsilon}(\xi_{0},\eta
)|+|K_{ij}^{\varepsilon}(\eta,\xi)-K_{ij}^{\varepsilon}(\eta,\xi_{0})|\\
&  \qquad\leq\boldsymbol{B}\,\frac{d(\xi_{0},\xi)}{d(\xi_{0},\eta)}\cdot
\frac{1}{|B_{d(\xi_{0},\eta)}(\xi_{0})|}%
\end{split}
\label{eq:standardiibToprove}%
\end{equation}
for every $\xi_{0},\xi,\eta\in\mathbb{R}^{N+1}$, provided that $d(\xi_{0}%
,\eta)>4\boldsymbol{\kappa}d(\xi_{0},\xi)$ (the choice $\beta=1$ follows from
the fact that $K_{ij}$ satisfies (ii) with $\beta=1$, see Theorem
\ref{thm:finepropGamma}). We limit ourselves to prove the above estimate for
the term
\[
|K_{ij}^{\varepsilon}(\xi,\eta)-K_{ij}^{\varepsilon}(\xi_{0},\eta)|
\]
since the other one is totally analogous.

To begin with, we observe that, by definition of $K_{ij}^{\varepsilon}$, we
have
\begin{align*}
&  |K_{ij}^{\varepsilon}(\xi,\eta)-K_{ij}^{\varepsilon}(\xi_{0},\eta)|\\
&  \qquad=|\phi_{\varepsilon}(t-s)\partial_{x_{i}x_{j}}^{2}\Gamma
(x,t;y,s)-\phi_{\varepsilon}(t_{0}-s)\partial_{x_{i}x_{j}}^{2}\Gamma
(x_{0},t_{0};y,s)|\\
&  \qquad\leq|\partial_{x_{i}x_{j}}^{2}\Gamma(\xi;\eta)-\partial_{x_{i}x_{j}%
}^{2}\Gamma(\xi_{0},\eta)|\,\phi_{\varepsilon}(t_{0}-s)\\
&  \qquad\qquad+\partial_{x_{i}x_{j}}^{2}\Gamma(\xi;\eta)|\phi_{\varepsilon
}(t_{0}-s)-\phi_{\varepsilon}(t-s)|\\
&  \qquad\equiv A_{1}+A_{2};
\end{align*}
moreover, since $0\leq\phi_{\varepsilon}\leq1$, from \eqref{eq:meanvalueGamma}
and \eqref{measure ball} we get%
\begin{equation}
A_{1}\leq\boldsymbol{B}\frac{d(\xi_{0},\xi)}{d(\xi_{0},\eta)^{Q+3}%
}=\boldsymbol{B}\frac{d(\xi_{0},\xi)}{d(\xi_{0},\eta)}\frac{1}{|B_{d(\xi
_{0},\eta)}(\xi_{0})|}. \label{eq:estimA1Gen}%
\end{equation}
Hence, we only need to estimate the term $A_{2}$. To this end we first notice
that, since $\phi_{\varepsilon}\in C_{0}^{\infty}(\mathbb{R})$, by the mean
value theorem we can write
\[
|\phi_{\varepsilon}(t_{0}-s)-\phi_{\varepsilon}(t-s)|=|\phi_{\varepsilon
}^{\prime}(\tau)|\,|t_{0}-t|;
\]
from this, since $\phi^{\prime}\not \equiv 0$ only on $(\varepsilon
,2\varepsilon)$ and since $|\phi_{\varepsilon}^{\prime}|\leq c/\varepsilon$
(for some absolute constant $c>0$), we obtain the following estimates
\begin{equation}%
\begin{split}
\mathrm{1)}\,\,  &  |\phi_{\varepsilon}(t_{0}-s)-\phi_{\varepsilon}(t-s)|\leq
c\,\frac{|t_{0}-t|}{\varepsilon}\leq c\,\frac{|t_{0}-t|}{|t-s|}\quad(\text{if
$\varepsilon<t-s<2\varepsilon$})\\[0.1cm]
\mathrm{2)}\,\,  &  |\phi_{\varepsilon}(t_{0}-s)-\phi_{\varepsilon}(t-s)|\leq
c\,\frac{|t_{0}-t|}{\varepsilon}\leq c\,\frac{|t_{0}-t|}{|t_{0}-s|}%
\quad(\text{if $\varepsilon<t_{0}-s<2\varepsilon$}).
\end{split}
\label{eq:twocasesEstimphie}%
\end{equation}
We then consider the two cases 1)\thinspace-\thinspace2) separately. (Note
that if none of these cases occurs, then $A_{2}=0$ and there is nothing to
prove).\medskip

\noindent-\thinspace\thinspace\textbf{Case 1):\thinspace\thinspace
$\varepsilon<t-s<2\varepsilon$}. In this first case, owing to
(\ref{eq:mainestimGaussian}) we get
\[
A_{2}\leq c\frac{|t_{0}-t|}{|t-s|}|\partial_{x_{i}x_{j}}^{2}\Gamma(\xi
,\eta)|\leq c\frac{|t_{0}-t|}{|t-s|^{2}}\Gamma_{c_{1}\nu^{-1}}(x,t;y,s)\leq
\frac{c|t_{0}-t|}{d(\xi,\eta)^{Q+4}};
\]
from this, using \eqref{eq:explicitd} (and since the condition $d(\xi_{0}%
,\eta)\geq4\bd\kappa\,d(\xi_{0},\xi)>0$ ensures that $d(\xi_{0},\eta)$ and
$d(\xi,\eta)$ are equivalent), we obtain
\[
A_{2}\leq c\frac{d\left(  \xi_{0},\xi\right)  ^{2}}{d\left(  \xi,\eta\right)
^{Q+4}}\leq c\frac{d\left(  \xi_{0},\xi\right)  }{d\left(  \xi_{0}%
,\eta\right)  ^{Q+3}}.
\]
By combining this last estimate with \eqref{eq:estimA1Gen}, we obtain
\eqref{eq:standardiibToprove}. \medskip

\noindent-\thinspace\thinspace\textbf{Case 2):\thinspace\thinspace
$\varepsilon<t_{0}-s<2\varepsilon$}. In this second case, we first write
\begin{align*}
&  |K_{ij}^{\varepsilon}(\xi,\eta)-K_{ij}^{\varepsilon}(\xi_{0},\eta)|\\
&  \qquad=|\phi_{\varepsilon}(t-s)\partial_{x_{i}x_{j}}^{2}\Gamma
(x,t;y,s)-\phi_{\varepsilon}(t_{0}-s)\partial_{x_{i}x_{j}}^{2}\Gamma
(x_{0},t_{0};y,s)|\\
&  \qquad\leq|\partial_{x_{i}x_{j}}^{2}\Gamma(\xi;\eta)-\partial_{x_{i}x_{j}%
}^{2}\Gamma(\xi_{0},\eta)|\,\phi_{\varepsilon}(t-s)\\
&  \qquad\qquad+\partial_{x_{i}x_{j}}^{2}\Gamma(\xi_{0};\eta)|\phi
_{\varepsilon}(t_{0}-s)-\phi_{\varepsilon}(t-s)|\\
&  \qquad\equiv B_{1}+B_{2}.
\end{align*}
Now, since $0\leq\phi_{\varepsilon}\leq1$, from \eqref{eq:meanvalueGamma} and
\eqref{measure ball} we get%
\begin{equation}
B_{1}\leq\boldsymbol{B}\frac{d(\xi_{0},\xi)}{d(\xi_{0},\eta)^{Q+3}%
}=\boldsymbol{B}\frac{d(\xi_{0},\xi)}{d(\xi_{0},\eta)}\frac{1}{|B_{d(\xi
_{0},\eta)}(\xi_{0})|}. \label{eq:estimB1Gen}%
\end{equation}
Moreover, by arguing exactly as in the previous case, we have
\begin{align*}
B_{2}  &  \leq c\frac{\left\vert t_{0}-t\right\vert }{|t_{0}-s|}%
|\partial_{x_{i}x_{j}}^{2}\Gamma(\xi_{0};\eta)|\leq c\frac{|t_{0}-t|}%
{|t_{0}-s|^{2}}\Gamma_{c_{1}\nu^{-1}}(x_{0},t_{0};y,s)\\
&  \leq\frac{c|t_{0}-t|}{d(\xi_{0},\eta)^{Q+4}}\leq c\frac{d(\xi_{0},\xi)^{2}%
}{d(\xi_{0},\eta)^{Q+4}}\leq c\frac{d(\xi_{0},\xi)}{d(\xi_{0},\eta)^{Q+3}}.
\end{align*}
By combining this last estimate with \eqref{eq:estimB1Gen}, we obtain
\eqref{eq:standardiibToprove} also in this case. \medskip

\noindent-\thinspace\thinspace\emph{Proof of property} (c). Owing to the
\emph{cancellation property} of $K_{ij}=\partial_{x_{i}x_{j}}^{2}\Gamma$
contain\-ed in Theorem \ref{thm:finepropGamma}\thinspace-\thinspace3), for
every $\zeta=(z,u)\in\mathbb{R}^{N+1}$ we have
\begin{align*}
&  \Big\vert\int_{\{r<d(\zeta,\xi)<R\}}K_{ij}^{\varepsilon}(\xi,\zeta
)d\xi\Big\vert+\Big\vert\int_{\{r<d(\zeta,\eta)<R\}}K_{ij}^{\varepsilon}%
(\zeta,\eta)d\eta\Big\vert\\
&  \qquad(\text{setting $\xi=(x,t),\,\eta=(y,s)$})\phantom{\int_a^b}\\
&  \qquad\leq\int_{u}^{u+R^{2}}\Big\vert\phi_{\varepsilon}(t-u)\Big(\int%
_{\{x\in\mathbb{R}^{N}:r<d((z,u),(x,t))<R\}}\partial_{x_{i}x_{j}}^{2}%
\Gamma(x,t;z,u)dx\Big)\Big\vert\,dt\\
&  \qquad\qquad+\int_{u-R^{2}}^{u}\Big\vert\phi_{\varepsilon}(u-s)\Big(\int%
_{\{y\in\mathbb{R}^{N}:r<d((z,u),(y,s))<R\}}\partial_{x_{i}x_{j}}^{2}%
\Gamma(z,u;y,s)dy\Big)\Big\vert\,ds\\
&  \qquad\leq\int_{u}^{u+R^{2}}\Big|\int_{\{x\in\mathbb{R}^{N}%
:r<d((z,u),(x,t))<R\}}\partial_{x_{i}x_{j}}^{2}\Gamma(x,t;z,u)dx\Big|\,dt\\
&  \qquad\qquad+\int_{u-R^{2}}^{u}\Big|\int_{\{y\in\mathbb{R}^{N}%
:r<d((z,u),(y,s))<R\}}\partial_{x_{i}x_{j}}^{2}\Gamma
(z,u;y,s)dy\Big\vert\,ds\\
&  \qquad\leq c,
\end{align*}
for some constant $c>0$ independent of $\zeta$.\medskip

This completes the proof of Lemma \ref{lem:PropKe} and therefore of Theorem
\ref{thm:LpEstimate}.
\end{proof}

We are now ready to prove the main result of this Section:

\begin{proof}
[Proof of Theorem \ref{Thm Krylov main step}]Given any $p\in(1,\infty)$, we
know from Theorem \ref{thm:LpEstimate} that the operator ${T}_{ij}$ defined in
\eqref{T} can be extended to a li\-near and continuous operator form
$L^{p}(\mathbb{R}^{N+1})$ into itself; moreover, by \eqref{repr formula u_xx}
we have
\[
\partial_{x_{i}x_{j}}^{2}u=T_{ij}(\overline{\mathcal{L}}u)
\]
for every $u\in C_{0}^{\infty}(\mathbb{R}^{N+1})$ and $1\leq i,j\leq q$.

Given any $f\in C_{0}^{\infty}(\mathbb{R}^{N+1})$, reasoning as in the proof
of Theorem \ref{thm:LpEstimate} we can write%
\[
T_{ij}(f)\left(  \xi\right)  =\lim_{\varepsilon\rightarrow0^{+}}\left(
-\int_{-\infty}^{t-\varepsilon}\int_{\mathbb{R}^{N}}\partial_{x_{i}x_{j}}%
^{2}\Gamma(\xi;\eta)f(\eta)\,d\eta\right)  .
\]
On the other hand, if $\xi\notin\mathrm{supp}(f)$, the last limit equals the
integral%
\[
-\int_{-\infty}^{t}\int_{\mathbb{R}^{N}}\partial_{x_{i}x_{j}}^{2}\Gamma
(\xi;\eta)f(\eta)\,d\eta,
\]
which is absolutely convergent. Therefore in this case%
\[
T_{ij}(f)\left(  \xi\right)  =\int_{\mathbb{R}^{N+1}}K_{ij}(\xi;\eta
)f(\eta)\,d\eta.
\]
Moreover, the kernel $K_{ij}$ satisfies the mean value inequality
(\ref{eq:meanvalueGamma}).

Finally, owing to \eqref{growth balls}, we have $\left\vert B_{kr}%
(\xi)\right\vert \leq c\,k^{Q+2}|B_{r}(\xi)|$, which is
(\ref{growth condition}) with $\beta=Q+2$.

Hence we can apply Theorem \ref{Thm abstract sing int} and conclude
(\ref{Krylov}). So the theorem is proved.
\end{proof}

\section{Operators with coefficients depending on $(x,t)$%
\label{Sec operators a(x,t)}}

With Theorem \ref{Thm Krylov main step} at hand, we can now prove Theorem
\ref{Thm main a priori estimates}. Henceforth we assume that $\mathcal{L}$ is
a KFP operator \eqref{L}, with coefficients $a_{ij}(x,t)$ \emph{depending on
both space and time} and fulfilling assumptions \textbf{(H1)}, \textbf{(H2)},
\textbf{(H3)} stated in Section \ref{sec intro}.

\subsection{Estimates on the mean oscillation of $\partial_{x_{i}x_{j}}^{2}u$
in terms of $\mathcal{L}u$\label{subsec mean oscillation}}

According to \cite{BB VMO}, the first step for the proof of Theorem
\ref{Thm main a priori estimates} consists in establishing a control on the
mean oscillation of $\partial_{x_{i}x_{j}}^{2}u$ for functions $u\in
C_{0}^{\infty}(\mathbb{R}^{N+1})$ \emph{with small support}, in terms of
$\mathcal{L}u$. To prove this result we combine Theorem
\ref{Thm Krylov main step} with the $VMO_{x}$ assumption on the $a_{ij}$'s,
following as far as possible Krylov' technique \cite{K}.

\begin{theorem}
\label{Thm 2} Let $p,\alpha,\beta\in(1,\infty)$ with $\alpha^{-1}+\beta
^{-1}=1$. Then, there exists a constant $c>0$, depending on $p,B,\nu$, such
that for every $R,r>0$, $\xi^{\ast},\xi_{0},\overline{\xi}\in\mathbb{R}^{N+1}$
with $\xi_{0}\in B_{r}(\overline{\xi})$ and $u\in C_{0}^{\infty}(B_{R}%
(\xi^{\ast}))$, and every $k\geq4\boldsymbol{\kappa}$, we have
\begin{equation}%
\begin{split}
&  \frac{1}{|B_{r}(\overline{\xi})|}\int_{B_{r}(\overline{\xi})}%
\big\vert\partial_{x_{i}x_{j}}^{2}u(\xi)-(\partial_{x_{i}x_{j}}^{2}%
u)_{B_{r}(\overline{\xi})}\big\vert\,d\xi\\
&  \qquad\leq\frac{c}{k}\sum_{h,l=1}^{q}\mathcal{M}(\partial_{x_{i}x_{j}}%
^{2}u)(\xi_{0})+ck^{\frac{Q+2}{p}}(\mathcal{M}(|\mathcal{L}u|^{p})(\xi
_{0}))^{1/p}\\
&  \qquad\qquad+ck^{\frac{Q+2}{p}}a^{\sharp}\left(  R\right)  ^{1/p\beta}%
\sum_{h,l=1}^{q}(\mathcal{M}(|\partial_{x_{h}x_{l}}^{2}u|^{p\alpha})(\xi
_{0}))^{1/p\alpha}%
\end{split}
\label{eq:meanEstimVMO}%
\end{equation}
for $i,j=1,2,...,q$. Recall that $a^{\sharp}\left(  R\right)  $ has been
defined in \eqref{mod VMO coeff}.
\end{theorem}

\begin{proof}
We can assume that $B_{r}(\overline{\xi})\cap B_{R}(\xi^{\ast})\neq
\varnothing$, because otherwise%
\[
\int_{B_{r}(\overline{\xi})}\big\vert\partial_{x_{i}x_{j}}^{2}u(\xi
)-(\partial_{x_{i}x_{j}}^{2}u)_{B_{r}(\overline{\xi})}\big\vert\,d\xi=0,
\]
and thus there is nothing to prove.

We then observe that, if $\overline{A}=(\gamma_{hl}(t))_{hl}$ is any fixed
matrix satisfying assumption \textbf{(H1)} (that is, $\gamma_{hl}\in
L^{\infty}(\mathbb{R}^{N+1})$ and the ellipticity condition \eqref{nu} holds),
by Theo\-rem \ref{Thm Krylov main step} we can write (provided that $k$ is
large enough)
\begin{equation}%
\begin{split}
&  \frac{1}{|B_{r}(\overline{\xi})|}\int_{B_{r}(\overline{\xi})}%
\big\vert\partial_{x_{i}x_{j}}^{2}u(\xi)-(\partial_{x_{i}x_{j}}^{2}%
u)_{B_{r}(\overline{\xi})}\big\vert\,d\xi\\
&  \qquad\leq c\Big\{\frac{1}{k}\mathcal{M}(\overline{\mathcal{L}}u)(\xi
_{0})\\
&  \qquad\qquad+k^{\frac{Q+2}{p}}\Big(\frac{1}{|B_{kr}(\overline{\xi})|}%
\int_{B_{kr}(\overline{\xi})}|\overline{\mathcal{L}}u(x,t)|^{p}%
\,dxdt\Big)^{1/p}\Big\},
\end{split}
\label{proof 2}%
\end{equation}
where $\overline{\mathcal{L}}$ is the KFP operator with coefficient matrix
$\overline{A}$, that is,
\begin{equation}
\textstyle\overline{\mathcal{L}}=\sum_{i,j=1}^{q}\gamma_{ij}(t)\partial
_{x_{i}x_{j}}^{2}+Y.\label{L bar}%
\end{equation}
We now turn to bound the right hand-side of \eqref{proof 2}. First, to handle
the term $\mathcal{M}\left(  \overline{\mathcal{L}}u\right)  $, let us write,
by (\ref{L bar}),%
\[
\overline{\mathcal{L}}u=\sum_{i,j=1}^{q}\gamma_{ij}(t)\partial_{x_{i}x_{j}%
}^{2}u+\mathcal{L}u-\sum_{i,j=1}^{q}a_{ij}(x,t)\partial_{x_{i}x_{j}}^{2}u
\]
so that%
\begin{equation}
\mathcal{M}\left(  \overline{\mathcal{L}}u\right)  \leq\mathcal{M}\left(
\mathcal{L}u\right)  +c\sum_{i,j=1}^{q}\mathcal{M}\left(  \partial_{x_{i}%
x_{j}}^{2}u\right)  \label{proof 2b}%
\end{equation}
with $c$ only depending on $\nu,q.$

To handle the second term at the right hand side of (\ref{proof 2}), we write%
\[
\Vert\overline{\mathcal{L}}u\Vert_{L^{p}(B_{kr}(\overline{\xi}))}\leq
\Vert\mathcal{L}u\Vert_{L^{p}(B_{kr}(\overline{\xi}))}+\Vert\overline
{\mathcal{L}}u-\mathcal{L}u\Vert_{L^{p}(B_{kr}(\overline{\xi}))}%
\]
and we exploit the fact that, since $\xi_{0}\in B_{r}(\overline{\xi})$, we
have
\[
\Big(\frac{1}{|B_{kr}(\overline{\xi})|}\int_{B_{kr}(\overline{\xi}%
)}|\mathcal{L}u(\xi)|^{p}\,d\xi\Big)^{1/p}\leq\big(\mathcal{M}(|\mathcal{L}%
u|^{p})(\xi_{0}))^{1/p};
\]
as a consequence of these facts, we get
\begin{equation}%
\begin{split}
&  \frac{1}{|B_{r}(\overline{\xi})|}\int_{B_{r}(\overline{\xi})}%
\big|\partial_{x_{i}x_{j}}^{2}u(\xi)-(\partial_{x_{i}x_{j}}^{2}u)_{B_{r}%
(\overline{\xi})}\big|\,d\xi\\
&  \qquad\leq\frac{c}{k}\left\{  \sum_{h,l=1}^{q}\mathcal{M}(\partial
_{x_{h}x_{l}}^{2}u)(\xi_{0})+\mathcal{M}\left(  \mathcal{L}u\right)  \left(
\xi_{0}\right)  \right\}  +ck^{\frac{Q+2}{p}}\big(\mathcal{M}(|\mathcal{L}%
u|^{p})(\xi_{0})\big)^{1/p}\\
&  \qquad\qquad+ck^{\frac{Q+2}{p}}\frac{1}{|B_{kr}(\overline{\xi})|^{1/p}%
}\Vert\overline{\mathcal{L}}u-\mathcal{L}u\Vert_{L^{p}(B_{kr}(\overline{\xi
}))}.
\end{split}
\label{proof 3}%
\end{equation}
Furthermore, we have (setting, as usual, $\xi=(x,t)$)
\begin{equation}%
\begin{split}
&  \int_{B_{kr}(\overline{\xi})}|\overline{\mathcal{L}}u(\xi)-\mathcal{L}%
u(\xi)|^{p}\,d\xi\\
&  \qquad\leq c\sum_{h,l=1}^{q}\Big(\int_{B_{kr}(\overline{\xi})\cap B_{R}%
(\xi^{\ast})}|\gamma_{hl}(t)-a_{hl}(x,t)|^{p\beta}\,dxdt\Big)^{1/\beta}%
\times\\
&  \qquad\qquad\times\Big(\int_{B_{kr}(\overline{\xi})\cap B_{R}(\xi^{\ast}%
)}|\partial_{x_{h}x_{l}}^{2}u(\xi)|^{p\alpha}\,d\xi\Big)^{1/\alpha}\\
&  \qquad(\text{since the coefficients $\gamma_{hl},\,a_{hl}$ are bounded by
$\nu^{-1}$})\\
&  \qquad\leq c(2\nu^{-1})^{p-1/\beta}\sum_{h,l=1}^{q}\Big(\int_{B_{kr}%
(\overline{\xi})\cap B_{R}(\xi^{\ast})}|{\gamma}_{hl}(t)-a_{hl}%
(x,t)|\,dxdt\Big)^{1/\beta}\times\\
&  \qquad\qquad\times\Big(\int_{B_{kr}(\overline{\xi})\cap B_{R}(\xi^{\ast}%
)}|\partial_{x_{h}x_{l}}^{2}u(\xi)|^{p\alpha}\,d\xi\Big)^{1/\alpha}.
\end{split}
\label{proof 4}%
\end{equation}
Now, since the above estimates hold \emph{for every fixed matrix}
$\overline{A}=(\gamma_{hl}(t))_{hl}$ satisfying assumption \textbf{(H1)}, we
choose the particular mat\-rix $(\gamma_{hl}(t))_{hl}$, depending on the
fix\-ed quantities $r,k,R,\xi^{\ast},\overline{\xi}$, defined as follows
\[
\gamma_{hl}(t)=%
\begin{cases}
(a_{hl}(\cdot,t))_{B_{R}(\xi^{\ast})} & \text{if }kr\geq R\\[0.1cm]%
(a_{hl}(\cdot,t))_{B_{kr}(\overline{\xi})} & \text{if }kr\leq R.
\end{cases}
\]
We recall that, according to Definition \ref{Def VMO}, we have
\[
(a_{hl}(\cdot,t))_{B}=\frac{1}{|B|}\int_{B}a_{hl}(x,t)\,dx\,ds\quad\text{for
every $d$-ball $B\subseteq\mathbb{R}^{N+1}$}.
\]
We then observe that, since the functions $a_{hl}\in L^{\infty}\left(
\mathbb{R}^{N+1}\right)  $, by Fubini's theorem we see that $\gamma_{hl}$ is
measurable for every $1\leq h,l\leq q$; on account of this fact, and since
$A_{0}=(a_{hl})_{hl}$ fulfills assumption \textbf{(H1)}, we easily conclude
that also the chosen matrix $\overline{A}=(\gamma_{hl})_{hl}$ satisfies
\textbf{(H1)}, and we can exploit the above estimates
\eqref{proof 3}\thinspace-\thinspace\eqref{proof 4} with this choice of $A$.
\vspace*{0.1cm}

As to estimate \eqref{proof 4}, we notice that by Definition \ref{Def VMO} we
have
\[
\int_{B_{kr}(\overline{\xi})\cap B_{R}(\xi^{\ast})}\big|a_{hl}(x,t)-\gamma
_{hl}(t)\big\vert\,dxdt\leq%
\begin{cases}
|B_{R}(\xi^{\ast})|a^{\sharp}\left(  R\right)   & \text{if $kr\geq R$%
}\\[0.15cm]%
|B_{kr}(\overline{\xi})|a^{\sharp}\left(  R\right)   & \text{if $kr\leqslant
R.$}%
\end{cases}
\]
On the other hand, since we are assuming that $B_{r}(\overline{\xi})$ and
$B_{R}(\xi^{\ast})$ intersect, in the case $kr\geq R$ we can write, by the
doubling condition,%
\[
\left\vert B_{R}(\xi^{\ast})\right\vert \leq\left\vert B_{kr}(\xi^{\ast
})\right\vert \leq c\left\vert B_{kr}\left(  \overline{\xi}\right)
\right\vert ;
\]
as a consequence, we obtain
\[
\int_{B_{kr}(\overline{\xi})\cap B_{R}(\xi^{\ast})}\big|a_{hl}(x,t)-\gamma
_{hl}(t)\big\vert\,dxdt\leq c|B_{kr}(\overline{\xi})|a^{\sharp}\left(
R\right)  .
\]
Summing up, estimate \eqref{proof 4} with this choice of $\overline{\emph{A}}$
boils down to%
\begin{equation}%
\begin{split}
&  \int_{B_{kr}(\overline{\xi})}|\overline{\mathcal{L}}u(\xi)-\mathcal{L}%
u(\xi)|^{p}\,d\xi\\
&  \qquad\leq c\,|B_{kr}(\overline{\xi})|^{1/\beta}a^{\#}\left(  R\right)
^{1/\beta}\sum_{h,l=1}^{q}\left(  \int_{B_{kr}(\overline{\xi})\cap B_{R}%
(\xi^{\ast})}|\partial_{x_{h}x_{l}}^{2}u(\xi)|^{p\alpha}\,d\xi\right)
^{1/\alpha}.
\end{split}
\label{eq:stimaprecisaconA}%
\end{equation}
With estimate \eqref{eq:stimaprecisaconA} at hand, we are now ready to
conclude the proof of the theorem: indeed, by combining (\ref{proof 2b}),
(\ref{proof 3}), \thinspace(\ref{eq:stimaprecisaconA}), we obtain%
\begin{equation}%
\begin{split}
&  \frac{1}{|B_{r}(\overline{\xi})|}\int_{B_{r}(\overline{\xi})}\left\vert
\partial_{x_{i}x_{j}}^{2}u(\xi)-(\partial_{x_{i}x_{j}}^{2}u)_{B_{r}%
(\overline{\xi})}\right\vert \,d\xi\\
&  \qquad\leq\frac{c}{k}\left\{  \sum_{i,j=1}^{q}\mathcal{M}(\partial
_{x_{i}x_{j}}^{2}u)(\xi_{0})+\mathcal{M}\left(  \mathcal{L}u\right)  \left(
\xi_{0}\right)  \right\}  +ck^{\frac{Q+2}{p}}\big(\mathcal{M}(|\mathcal{L}%
u|^{p})(\xi_{0})\big)^{1/p}\\
&  \qquad\qquad+ck^{\frac{Q+2}{p}}|B_{kr}(\overline{\xi})|^{\frac{1}{p\beta
}-\frac{1}{p}}a^{\sharp}\left(  R\right)  ^{1/p\beta}\times\\
&  \qquad\qquad\qquad\times\sum_{h,l=1}^{q}\Big(\int_{B_{kr}(\overline{\xi
})\cap B_{R}(\xi^{\ast})}|\partial_{x_{h}x_{l}}^{2}u(\xi)|^{p\alpha}%
\,d\xi\Big)^{1/p\alpha}%
\end{split}
\label{eq:toboundlastline}%
\end{equation}
furthermore, recalling that $\alpha^{-1}+\beta^{-1}=1$ and using
\eqref{maximal HL} (note that $\xi_{0}\in B_{kr}(\overline{\xi})$), we can
bound the the last line of \eqref{eq:toboundlastline} as
\begin{equation}%
\begin{split}
&  |B_{kr}(\overline{\xi})|^{\frac{1}{p\beta}-\frac{1}{p}}a^{\sharp}\left(
R\right)  ^{1/p\beta}\sum_{h,l=1}^{q}\Big(\int_{B_{kr}(\overline{\xi})\cap
B_{R}(\xi^{\ast})}|\partial_{x_{h}x_{l}}^{2}u(\xi)|^{p\alpha}\,d\xi
\Big)^{1/p\alpha}\\
&  \qquad=a^{\sharp}\left(  R\right)  ^{1/p\beta}\sum_{h,l=1}^{q}\Big(\frac
{1}{|B_{kr}(\overline{\xi})|}\int_{B_{kr}(\overline{\xi})\cap B_{R}(\xi^{\ast
})}|\partial_{x_{h}x_{l}}^{2}u(\xi)|^{p\alpha}\,d\xi\Big)^{1/p\alpha}\\
&  \qquad\leq a^{\sharp}\left(  R\right)  ^{1/p\beta}\sum_{h,l=1}%
^{q}\big(\mathcal{M}(|\partial_{x_{i}x_{j}}^{2}u|^{p\alpha})(\xi
_{0})\big)^{1/p\alpha},
\end{split}
\label{eq:boundlastline}%
\end{equation}
for some $c$ also depending on $\nu.$ Note also that in the right-hand side of
(\ref{eq:toboundlastline}), the term%
\[
\frac{c}{k}\mathcal{M}\left(  \mathcal{L}u\right)  \left(  \xi_{0}\right)
\text{ can be absorbed in }ck^{\frac{Q+2}{p}}\left(  \left(  \mathcal{M}%
\left(  \left\vert \mathcal{L}u\right\vert ^{p}\right)  \right)  \left(
\xi_{0}\right)  \right)  ^{1/p}\text{ }%
\]
since $k$ is large and for every ball $B_{\rho}\ni\xi_{0}$ we can write%
\[
\frac{1}{\left\vert B_{\rho}\right\vert }\int_{B_{\rho}}\left\vert
\mathcal{L}u\left(  x\right)  \right\vert dx\leq\left(  \frac{1}{\left\vert
B_{\rho}\right\vert }\int_{B_{\rho}}\left\vert \mathcal{L}u\left(  x\right)
\right\vert ^{p}dx\right)  ^{1/p}\leq\left(  \left(  \mathcal{M}\left(
\left\vert \mathcal{L}u\right\vert ^{p}\right)  \right)  \left(  \xi
_{0}\right)  \right)  ^{1/p},
\]
hence%
\begin{equation}
\mathcal{M}\left(  \mathcal{L}u\right)  \left(  \xi_{0}\right)  \leq\left(
\left(  \mathcal{M}\left(  \left\vert \mathcal{L}u\right\vert ^{p}\right)
\right)  \left(  \xi_{0}\right)  \right)  ^{1/p}.\label{proof 5}%
\end{equation}

Gathering \eqref{eq:toboundlastline}\thinspace, \thinspace
\eqref{eq:boundlastline} and (\ref{proof 5}), we finally obtain the desired
\eqref{eq:meanEstimVMO}, and the proof is complete.
\end{proof}

\subsection{$L^{p}$ estimates\label{subsec Lp}}

With Theorem \ref{Thm 2} at hand, we can easily prove \emph{local $L^{p}$
esti\-ma\-tes} for functions $u\in C_{0}^{\infty}(\mathbb{R}^{N+1})$ with
\emph{small support}:

\begin{theorem}
\label{Thm local small Lp}For every $p\in(1,\infty)$ there exist constants
$R,c>0$ such that, for every ball $B_{R}(\xi^{\ast})$ in $\mathbb{R}^{N+1}$
and every $u\in C_{0}^{\infty}(B_{R}(\xi^{\ast}))$, we have
\begin{equation}
\sum_{h,l=1}^{q}\left\Vert u_{x_{h}x_{l}}\right\Vert _{L^{p}\left(
B_{R}\left(  \xi^{\ast}\right)  \right)  }\leq c\left\Vert \mathcal{L}%
u\right\Vert _{L^{p}\left(  B_{R}\left(  \xi^{\ast}\right)  \right)  }.
\label{final local}%
\end{equation}
The constants $c,R$ in \eqref{final local} depends on the numbers $p,\nu,B$
and on the function $a^{\sharp}$ in \eqref{mod VMO coeff}, but do not depend
on $\xi^{\ast}.$
\end{theorem}

\begin{proof}
The present theorem can be established by arguing \emph{exactly} as in the
proof of \cite[Thm.\,4.2]{BB VMO}; the idea is to combine the mean-oscillation
estima\-tes in The\-orem \ref{Thm 2} (which is the analog of \cite[Thm.\,4.1]%
{BB VMO}) with the Hardy-Lit\-tle\-wood maximal inequality and the
Fefferman-Stein-like maximal inequality for spa\-ces of ho\-mo\-ge\-neo\-us
type (see Theorem \ref{Thm Maximal} and Theorem \ref{Def sharp},
respectively), and to exploit as\-sumption \textbf{(H3)} to take to the left
hand side the term involving $a^{\sharp}\left(  R\right)  $ in \eqref{eq:meanEstimVMO}.
\end{proof}

Starting from Theorem \ref{Thm local small Lp}, we can then establish the
following global a priori e\-sti\-mates, which are a `less refined' version of
Theorem \ref{Thm main a priori estimates}.

\begin{theorem}
\label{Thm stima globale ver1} For every $p\in(1,\infty)$ there exists $c>0$,
depending on $p,\nu,$ $B$ and on the function $a_{\cdot}^{\sharp}$, such that,
for every $u\in C_{0}^{\infty}(\mathbb{R}^{N+1})$, one has
\[
\left\Vert u\right\Vert _{W_{X}^{2,p}(\mathbb{R}^{N+1})}\leq c\big\{\left\Vert
\mathcal{L}u\right\Vert _{L^{p}(\mathbb{R}^{N+1})}+\left\Vert u\right\Vert
_{W_{X}^{1,p}(\mathbb{R}^{N+1})}\big\}.
\]
Moreover, for every $R>0$ and every $u\in C^{\infty}(\mathbb{R}^{N+1})$ we
have%
\[
\left\Vert u\right\Vert _{W_{X}^{2,p}(B_{R}(0))}\leq c_{1}\big\{\left\Vert
\mathcal{L}u\right\Vert _{L^{p}(B_{2R}(0))}+\left\Vert u\right\Vert
_{W_{X}^{1,p}(B_{2R}(0))}\big\},
\]
where $c_{1}>0$ is a constant independent of $R$.
\end{theorem}

According to \cite{BB VMO}, in order to prove the above result we need to show
the existence of suitable cutoff functions and of a suitable covering of
$\mathbb{R}^{N+1}$. \medskip

We begin with the existence of an ad-hoc family of cutoff functions.

\begin{lemma}
\label{Lemma cutoff} For every fixed $R>0$, there exist a constant $c>0$ and a
fa\-mi\-ly $\{\phi^{\xi}(\cdot)\}_{\xi\in\mathbb{R}^{N+1}}$ of cutoff
functions in $\mathbb{R}^{N+1}$ such that%
\begin{align*}
\mathrm{i)}  &  \,\,\phi^{\xi}\in C_{0}^{\infty}(B_{2R}(\xi))\\
\mathrm{ii)}  &  \,\,\text{$\phi^{\xi}=1$ on $B_{R}(\xi)$ and $0\leq\phi^{\xi
}\leq1$ on $\mathbb{R}^{N+1}$};\\
\mathrm{iii)}  &  \,\,\sup_{\eta}|\phi^{\xi}(\eta)|+\sum_{i=1}^{q}\sup_{\eta
}|\partial_{y_{i}}\phi^{\xi}(\eta)|+\sum_{i,j=1}^{q}\sup_{\eta}|\partial
_{y_{i}y_{j}}^{2}\phi^{\xi}(\eta)|+\sup_{\eta}|Y_{y}\phi^{\xi}(\eta)|\leq c
\end{align*}
\emph{(}where $\eta=(y,s)$\emph{)}. Here, the relevant fact is that the
constant $c$ is independent of $\xi$ \emph{(}while it may depend on
$R$\emph{)}.
\end{lemma}

\begin{proof}
Let $\psi\in C_{0}^{\infty}(\mathbb{R}^{N+1})$ be a fixed cutoff function such that

\begin{itemize}
\item[a)] $0\leq\psi\leq1$ on $\mathbb{R}^{N+1}$;

\item[b)] $\psi= 1$ on $B_{1}(0)$ and $\psi= 0$ out of $B_{2}(0)$
\end{itemize}

Given any $\xi\in\mathbb{R}^{N+1}$, we define the function
\[
\phi^{\xi}(\eta)=\psi\big(D(1/R)(\xi^{-1}\circ\eta)\big),
\]
and we claim that it satisfies the required properties i)\thinspace-\thinspace iii).

Clearly, we have $\psi^{\xi}\in C_{0}^{\infty}(\mathbb{R}^{N+1})$ and
$0\leq\psi^{\xi}\leq1$; moreover, using \eqref{d}-\eqref{eq:balltraslD} (and
taking into account property b) of $\psi$), we get

\begin{itemize}
\item $\phi^{\xi}= 1$ on $\xi\circ D(R)(B_{1}(0)) = B_{R}(\xi)$;

\item $\psi^{\xi}= 0$ out of $\xi\circ D(R)(B_{2}(0)) = B_{2R}(\xi)$.
\end{itemize}

Finally, recalling that $\partial_{y_{i}}$ (with $1\leq i\leq q$) and $Y$ are
left-invariant and $D(\lambda)$-ho\-mo\-geneous of degree $1$ and $2$,
respectively, for every $\eta\in\mathbb{R}^{N+1}$ we obtain
\begin{align*}
&  |\phi^{\xi}(\eta)|+\sum_{i=1}^{q}|\partial_{y_{i}}\phi^{\xi}(\eta
)|+\sum_{i,j=1}^{q}|\partial_{y_{i}y_{j}}^{2}\phi^{\xi}(\eta)|+|Y_{y}\phi
^{\xi}(\eta)|\\
&  \qquad(\text{setting $\zeta=D(1/R)(\xi^{-1}\circ\eta)$})\\
&  \qquad=|\psi(\zeta)|+R^{-1}\sum_{i=1}^{q}|(\partial_{i}\psi)(\zeta
)|+R^{-2}\Big(\sum_{i,j=1}^{q}|(\partial_{ij}^{2}\psi)(\zeta)|+|(Y\psi
)(\zeta)|\Big)\\
&  \qquad\leq c,
\end{align*}
where $c>0$ depends on $\psi$ and $R$. This ends the proof.
\end{proof}

As for the existence of a suitable covering of $\mathbb{R}^{N+1}$, instead, we
recall the fol\-low\-ing \emph{co\-ve\-ring theorem in spaces of homogeneous
type}.

\begin{proposition}
[{See \cite[Prop.\,2.19]{BB VMO}}]\label{Prop overlapping} Let $(X,d,\mu)$ be
a space of homogeneous type. For every fixed $R>0$ there exists a family
\[
\mathcal{B}=\{B(x_{\alpha},R)\}_{\alpha\in A}%
\]
of $d$-balls in $\mathbb{R}^{N+1}$ satisfying the following properties:

\begin{itemize}
\item[i)] $\bigcup_{\alpha\in A}B(x_{\alpha},R) =X$; \vspace*{0.1cm}

\item[ii)] for every $H>1$, the family of dilated balls%
\[
\mathcal{B}^{H}=\{ B(x_{\alpha},HR)\} _{\alpha\in A}
\]
has the \emph{bounded overlapping property}, that is, there exists a constant
$N$ \emph{(}depending on $H$ and the constants of $X$, but independent of
$R$\emph{)} such that every point of $X$ belongs to at most $N$ balls
$B(x_{\alpha},HR) $.
\end{itemize}
\end{proposition}

With Lemma \ref{Lemma cutoff} and Proposition \ref{Prop overlapping} at hand,
we can provide the

\begin{proof}
[Proof of Theorem \ref{Thm stima globale ver1}]The present theorem can be
established by arguing exactly as in the proof of \cite[Thm.\,4.3]{BB VMO} and
\cite[Cor.\,4.4]{BB VMO}; the idea is to apply the local $L^{p}$ estimates in
Theorem \ref{Thm local small Lp} to the function $u\phi^{\xi}$ (where
$\phi^{\xi}$ is as in Lemma \ref{Lemma cutoff}), and then to exploit
Proposition \ref{Prop overlapping}.
\end{proof}

Thanks to Theorem \ref{Thm stima globale ver1}, we can finally come to the

\begin{proof}
[Proof of Theorem \ref{Thm main a priori estimates}]Taking into account all
the results established so far, the present theorem can be demonstrated by
arguing exactly as in the proof of \cite[Thm. 1.4]{BB VMO}: first of all, one
extends the estimate in Theorem \ref{Thm stima globale ver1} to functions
$u\in W_{X}^{2,p}(\mathbb{R}^{N+1})$ by using the \emph{local approximation
of} $W_{X}^{2,p}(\mathbb{R}^{N+1})$ \emph{by smooth functions} (see
\cite[Thm.\,2.9]{BBbook}); then, one gets rid of the term
\[
\Vert u\Vert_{W_{X}^{1,p}(\mathbb{R}^{N+1})}%
\]
by exploiting the following known (Euclidean!) interpolation inequality,
\begin{equation}
\label{eq:interpIneq}\Vert\partial_{x_{i}}u\Vert_{L^{p}(\mathbb{R}^{N+1})}%
\leq\varepsilon\Vert\partial_{x_{i}x_{i}}^{2}u\Vert_{L^{p}(\mathbb{R}^{N+1}%
)}+\frac{c_{p}}{\varepsilon}\Vert u\Vert_{L^{p}(\mathbb{R}^{N+1})},
\end{equation}
holding true for every $p\in(1,\infty)$, every $u\in W_{X}^{2,p}%
(\mathbb{R}^{N+1})$, every $1\leq i,j\leq q$ and every $\varepsilon>0$, with a
constant $c_{p}>0$ only depending on $p$.
\end{proof}

\section{Existence results\label{Sec existence}}

In this last section we exploit the global $W_{X}^{2,p}$\thinspace-\thinspace
estimates contained in Theorem \ref{Thm main a priori estimates} to prove some
existence results for the equation
\[
\text{$\mathcal{L}u-\lambda u=f$ in $S_{T}=\mathbb{R}^{N}\times(-\infty,T),$}%
\]
with $f\in L^{p}\left(  S_{T}\right)  $, for $\lambda>0$ large enough, and
every $T\in(-\infty,+\infty]$, and for the Cauchy problem%
\[
\mathrm{(CP)}\qquad\quad%
\begin{cases}
\mathcal{L}u=f & \text{in $\mathbb{R}^{N}\times(0,T)$}\\
u(\cdot,0)=g & \text{in $\mathbb{R}^{N}$}%
\end{cases}
\]
with $f\in L^{p}\left(  \text{$\mathbb{R}^{N}\times(0,T)$}\right)  $, $g\in
W_{X}^{2,p}\left(  \mathbb{R}^{N}\right)  $ for every $T\in\left(
0,+\infty\right)  $ (see Section \ref{Subsec Cauchy} for the precise
definition of solution). In both cases, the solutions we obtain are actually
\emph{strong solutions}, that is, they be\-long to some $W_{X}^{2,p}%
$\thinspace-\thinspace space, and the equation $\mathcal{L}u=f$ holds
pointwise (a.e.).

In order to prove our results we mainly follow the approach by Krylov
\cite{KrylovBook} for parabolic equations, which essentially consists in the
following steps.

\begin{enumerate}
\item First of all, we refine the $W_{X}^{2,p}$\thinspace-\thinspace estimates
in Theorem \ref{Thm main a priori estimates} in two different directions: we
localize such estimates on any strip $S_{T}=\mathbb{R}^{N}\times(-\infty,T)$,
and we get rid of the term $\Vert u\Vert_{L^{p}}$ on the right-hand side (by
paying the price of considering the `perturbed operator' $\mathcal{L}-\lambda$
for some $\lambda>0$ large enough). \vspace*{0.1cm}

\item Using these refined estimates and the method of continuity, we then
prove the existence of a unique $W_{X}^{2,p}(S_{T})$\thinspace-\thinspace
solution of the equation
\[
\mathcal{L}u-\lambda u=f\quad\text{in $S_{T}$},
\]
for an arbitrary $f\in L^{p}\left(  \mathbb{R}^{N+1}\right)  $ (and for every
$-\infty<T\leq+\infty$). \vspace*{0.1cm}

\item Finally, using the above solvability result (and the refined esti\-mates
in point (1)), we establish the existence of a unique solution of the Cauchy
problem (CP).
\end{enumerate}

Throughout this section, we adopt the notation:

\begin{itemize}
\item given any $T\in(-\infty,+\infty]$, we set $S_{T}=\mathbb{R}^{N}%
\times(-\infty,T)$; \vspace*{0.1cm}

\item given any $T\in\left(  0,+\infty\right)  $, we set $\Omega
_{T}=\mathbb{R}^{N}\times\left(  0,T\right)  $.
\end{itemize}

\subsection{\noindent Refined $W_{X}^{2,p}$\thinspace-\thinspace
estimates\label{subsec refined}}

We start with the following theorem, localizing the global estimates of
Theorem \ref{Thm main a priori estimates} on any strip $S_{T}$.

\begin{theorem}
\label{Thm global strip} Let $\mathcal{L}$ be an operator as in \eqref{L},
assume that \emph{\textbf{(H1)}, \textbf{(H2)}, \textbf{(H3)}} hold, and let
$p\in(1,\infty)$.

Then, there exists a constant $c>0$ \emph{(}depending on $p$, the matrix $B$
in \eqref{B}, the number $\nu$ in \eqref{nu}, and the function $a^{\#}$ in
\eqref{mod VMO coeff}\emph{)} such that, for every $T\in(-\infty,+\infty]$,%
\[
\left\Vert u\right\Vert _{W_{X}^{2,p}\left(  S_{T}\right)  }\leq c\left\{
\Vert\mathcal{L}u\Vert_{L^{p}\left(  S_{T}\right)  }+\Vert u\Vert
_{L^{p}\left(  S_{T}\right)  }\right\}
\]
{for every }$u\in W_{X}^{2,p}\left(  S_{T}\right)  $. Note that the constant
$c$ does not depend on $T$.
\end{theorem}

The proof of Theorem \ref{Thm global strip} follows by carefully revising the
one of Theorem \ref{Thm main a priori estimates}, and by exploiting the
following proposition.

\begin{proposition}
\label{prop:balltagliate}There exists $c>0$ such that if, for every $\xi\in
S_{T}$ and $r>0$, we let%
\[
B_{r}^{T}(\xi_{0})=B_{r}(\xi_{0})\cap S_{T},
\]
then;
\begin{align*}
\mathrm{i)}  &  \,\,\left\vert B_{2r}^{T}\left(  \xi\right)  \right\vert \leq
c\left\vert B_{r}^{T}(\xi)\right\vert \\
\mathrm{ii)}  &  \,\,\left\vert B_{kr}^{T}(\xi)\right\vert \leq ck^{Q+2}%
\left\vert B_{r}^{T}(\xi)\right\vert \text{ for every }k\geq1\text{.}%
\end{align*}

\end{proposition}

\begin{proof}
We will show that, for every $\xi\in S_{T}$ and $r>0$,%
\begin{equation}
\left\vert B_{r}^{T}\left(  \xi\right)  \right\vert \geq\frac{1}{2}\left\vert
B_{r}\left(  \xi\right)  \right\vert . \label{sfere tagliate}%
\end{equation}
This implies
\[
\left\vert B_{2r}^{T}\left(  \xi\right)  \right\vert \leq\left\vert
B_{2r}\left(  \xi\right)  \right\vert \leq c\left\vert B_{r}\left(
\xi\right)  \right\vert \leq2c\left\vert B_{r}^{T}\left(  \xi\right)
\right\vert
\]
and also%
\[
\left\vert B_{kr}^{T}\left(  \xi\right)  \right\vert \leq\left\vert
B_{kr}\left(  \xi\right)  \right\vert \leq ck^{Q+2}\left\vert B_{r}\left(
\xi\right)  \right\vert \leq2ck^{Q+2}\left\vert B_{r}^{T}\left(  \xi\right)
\right\vert .
\]
To prove (\ref{sfere tagliate}), let us write, for $\xi=\left(  x,t\right)
,\eta=\left(  y,s\right)  $%
\[
\left\vert B_{r}^{T}\left(  \xi\right)  \right\vert =\int_{\rho\left(
\eta^{-1}\circ\xi\right)  <r,s<T}d\xi
\]
letting $\eta^{-1}\circ\xi=\zeta=\left(  z,\sigma\right)  ,$ with $\sigma=t-s$%
\[
=\int_{\rho\left(  \zeta\right)  <r,\sigma>t-T}d\zeta
\]
since $\xi\in S_{T}$, we have $t-T<0$, hence%
\begin{align*}
&  \geq\int_{\rho\left(  \zeta\right)  <r,\sigma>0}d\zeta=\int_{\left\Vert
z\right\Vert +\sqrt{\left\vert \sigma\right\vert }<r,\sigma>0}dzd\sigma\\
&  =\frac{1}{2}\int_{\left\Vert z\right\Vert +\sqrt{\left\vert \sigma
\right\vert }<r}dzd\sigma=\frac{1}{2}\left\vert B_{r}\left(  \xi\right)
\right\vert .
\end{align*}
This ends the proof.
\end{proof}

With Proposition \ref{prop:balltagliate} at hand, we can provide the

\begin{proof}
[Proof of Theorem \ref{Thm global strip}]First of all we observe that, as we
read from Proposition \ref{prop:balltagliate}, $\left(  S_{T},d,\left\vert
\cdot\right\vert \right)  $ is a space of homogeneous type; in particular, the
Hardy-Littlewood maximal operator $\mathcal{M}^{T}$ (defined on $S_{T}$
w.r.t.\thinspace the balls $B_{r}^{T}$) maps the space $L^{p}\left(
S_{T}\right)  $ into itself, and the same is true for the analogous sharp
maximal operator (since $S_{T}$ is still unbounded). Moreover, if $T_{ij}$ is
the singular\thinspace-\thinspace in\-tegral operator defined w.r.t.\thinspace
the differential operator $\overline{\mathcal{L}}$ (see \eqref{T}), we have:%
\[
T_{ij}\left(  f\right)  =T_{ij}\left(  f\cdot\mathbf{1}_{S_{T}}\right)  \text{
in }S_{T}\text{.}%
\]
As a consequence, we have
\[
\left\Vert T_{ij}\left(  f\right)  \right\Vert _{L^{p}\left(  S_{T}\right)
}\leq c\left\Vert f\right\Vert _{L^{p}\left(  S_{T}\right)  }\text{ for every
}f\in L^{p}\left(  S_{T}\right)  \text{ and }p\in\left(  1,\infty\right)
\text{.}%
\]
Gathering all these facts, and revising carefully the proofs of the results we
have established on the entire space $\mathbb{R}^{N+1}$, we can then infer the
following facts. \medskip

a)\,\,Theorem \ref{Thm Krylov main step} holds with
\[
B_{r}\mapsto B_{r}^{T};\qquad\mathcal{M\mapsto M}^{T}.
\]

b)\thinspace\thinspace In the proof of Theorem \ref{Thm 2}, we can replace
\[
B_{r}\mapsto B_{r}^{T};\qquad\mathcal{M\mapsto M}^{T}.
\]
Thus, we have
\begin{align*}
&  \int_{B_{kr}^{T}(\overline{\xi})\cap B_{R}^{T}(\xi^{\ast})}\left\vert
\gamma_{hl}\left(  t\right)  -a_{hl}\left(  x,t\right)  \right\vert dxdt\\
&  \qquad\leq\int_{B_{kr}(\overline{\xi})\cap B_{R}(\xi^{\ast})}\left\vert
\gamma_{hl}\left(  t\right)  -a_{hl}\left(  x,t\right)  \right\vert dxdt\\
&  \qquad(\text{reasoning as in the proof of Theorem \ref{Thm 2}})\\
&  \qquad\leq c\left\vert B_{kr}(\overline{\xi})\right\vert a^{\#}\left(
R\right)  \leq2c\left\vert B_{kr}^{T}(\overline{\xi})\right\vert a^{\#}\left(
R\right)  .
\end{align*}
and the proof can be concluded at the same way. \vspace*{0.1cm}

c)\,\,On account of a),\-\,b), Theorem \ref{Thm local small Lp} can be proved
with $B_{R}\mapsto B_{R}^{T}$, by arguing exactly in the same way and
exploiting the continuity of the maximal Hardy Littlewood and sharp maximal
operators over $S_{T}$. \vspace*{0.1cm}

\noindent In view of these facts, and since the interpolation inequalities
\eqref{eq:interpIneq} hold for every fixed time (hence, they also hold on the
strip $S_{T}$), we can conclude the proof of Theorem \ref{Thm global strip} by
proceeding as in the proof of Theorem \ref{Thm main a priori estimates}.
\end{proof}

Now that we have established Theorem \ref{Thm global strip}, we proceed by
proving the second impro\-vement of our estimates announced at the beginning
of the section.

\begin{theorem}
\label{thm:ImprovedLambda} Let $\mathcal{L}$ be an operator as in \eqref{L},
assume that \emph{\textbf{(H1)}, \textbf{(H2)}, \textbf{(H3)}} hold, and let
$p\in\left(  1,\infty\right)  $. Then, there exist positive constants
$c,\lambda_{0}$ \emph{(}depending on $p$, the matrix $B$ in \eqref{B}, the
number $\nu$ in \eqref{nu}, and the function $a^{\#}$ in
\eqref{mod VMO coeff}\emph{)} such that%
\begin{equation}
\left\Vert u\right\Vert _{W_{X}^{2,p}\left(  S_{T}\right)  }\leq
c\Vert\mathcal{L}u-\lambda u\Vert_{L^{p}\left(  S_{T}\right)  }
\label{eq:LpOmogenee}%
\end{equation}
for every $T\in(-\infty,+\infty]$, $u\in W_{X}^{2,p}(S_{T})$ and $\lambda
\geq\lambda_{0}$. Note that the constants $c,\lambda_{0}$ do not depend on $T$.
\end{theorem}

It should be noticed that the appearence of the $L^{p}$\thinspace-\thinspace
norm of $\mathcal{L}_{\lambda}u$ (instead of that of $\mathcal{L}$) in the
right-hand side of \eqref{eq:LpOmogenee} is somehow unavoidable: indeed, in
the special case when $\mathcal{L}$ has \emph{constant coefficients} and
$T=+\infty$, a simple homogeneity argument shows that an estimate of the form
\eqref{eq:LpOmogenee} \emph{cannot hold} for $\mathcal{L}$.

\begin{proof}
To ease the readability (and to simplify the notation), we only consider the
case $T=+\infty$ (the case $T<+\infty$ being completely analogous). To begin
with, let $\widetilde{\mathcal{L}}$ be the KFP operator defined on
$\mathbb{R}^{N+2}=\mathbb{R}_{y}\times\mathbb{R}_{(x,t)}^{N+1}$ as
\[
\widetilde{\mathcal{L}}=\mathcal{L+\partial}_{yy}^{2}=\partial_{yy}^{2}%
+\sum_{i,j=1}^{q}a_{ij}(x,t)\partial_{x_{i}x_{j}}^{2}+Y.
\]
Clearly, this operator $\widetilde{\mathcal{L}}$ satisfies assumptions
\textbf{(H1)}\thinspace-\thinspace\textbf{(H3)}, with
\[
\widetilde{X}=\{\partial_{y},\partial_{x_{1}},\ldots,\partial_{x_{n}%
},Y\},\qquad\widetilde{A}_{0}(x,t)=%
\begin{pmatrix}
1 & \mathbb{O}\\
\mathbb{O} & A_{0}(x,t)
\end{pmatrix}
;
\]
hence, we can apply Theorem \ref{Thm main a priori estimates} to
$\widetilde{\mathcal{L}}$, obtaining the estimate
\begin{equation}
\Vert u\Vert_{W_{\widetilde{X}}^{2,p}(\mathbb{R}^{N+2})}\leq c\big\{\Vert
\widetilde{\mathcal{L}}u\Vert_{L^{p}(\mathbb{R}^{N+2})}+\Vert u\Vert
_{L^{p}(\mathbb{R}^{N+2})}\big\}, \label{global u}%
\end{equation}
which holds for every function $u=u(y,x,t)\in W_{\widetilde{X}}^{2,p}%
(\mathbb{R}^{N+2})$.

Let now $\phi\left(  y\right)  \in C_{0}^{\infty}\left(  -1,1\right)
,\,\phi\not \equiv 0$, and let $u\left(  x,t\right)  \in W_{X}^{2,p}\left(
\mathbb{R}^{N+1}\right)  $. We set
\[
\widetilde{u}\left(  y,x,t\right)  =u\left(  x,t\right)  \phi\left(  y\right)
e^{i\sqrt{\lambda}y}\in W_{X}^{2,p}\left(  \mathbb{R}^{N+2}\right)  .
\]
We can apply the bound (\ref{global u}) to $\widetilde{u}$ and
$\widetilde{\mathcal{L}}$ on $\mathbb{R}^{N+2}$ (since $\widetilde{\mathcal{L}%
}$ is linear with real coefficients, the bound extends to complex valued
functions):%
\begin{equation}
\left\Vert \widetilde{u}\right\Vert _{W_{\widetilde{X}}^{2,p}\left(
\mathbb{R}^{N+2}\right)  }\leq c\left\{  \Vert\widetilde{\mathcal{L}%
}\widetilde{u}\Vert_{L^{p}\left(  \mathbb{R}^{N+2}\right)  }+\Vert
\widetilde{u}\Vert_{L^{p}\left(  \mathbb{R}^{N+2}\right)  }\right\}  .
\label{1}%
\end{equation}
Now, by definition of $\widetilde{u}$,%
\[
\left\Vert \widetilde{u}\right\Vert _{W_{{\widetilde{X}}}^{2,p}\left(
\mathbb{R}^{N+2}\right)  }\geq c\left\Vert u\right\Vert _{W_{X}^{2,p}\left(
\mathbb{R}^{N+1}\right)  }+c\left\Vert u\right\Vert _{L^{p}\left(
\mathbb{R}^{N+1}\right)  }\left\Vert \phi^{\prime\prime}+2i\sqrt{\lambda}%
\phi^{\prime}-\lambda\phi\right\Vert _{L^{p}\left(  -1,1\right)  }%
\]
and%
\[
\left\Vert \phi^{\prime\prime}+2i\sqrt{\lambda}\phi^{\prime}-\lambda
\phi\right\Vert _{L^{p}\left(  -1,1\right)  }\geq\left\Vert \phi^{\prime
\prime}-\lambda\phi\right\Vert _{L^{p}\left(  -1,1\right)  }\geq c\lambda
\]
for $\lambda$ large enough depending on $\phi$, and $c>0$ depending on $\phi$.
Hence%
\begin{equation}
\left\Vert \widetilde{u}\right\Vert _{W_{X}^{2,p}\left(  \mathbb{R}%
^{N+2}\right)  }\geq c\left\{  \left\Vert u\right\Vert _{W_{X}^{2,p}\left(
\mathbb{R}^{N+1}\right)  }+\lambda\left\Vert u\right\Vert _{L^{p}\left(
\mathbb{R}^{N+1}\right)  }\right\}  . \label{2}%
\end{equation}

Also,%
\[
\Vert\widetilde{u}\Vert_{L^{p}\left(  \mathbb{R}^{N+2}\right)  }\leq c\Vert
u\Vert_{L^{p}\left(  \mathbb{R}^{N+1}\right)  }%
\]
and%
\begin{align*}
\widetilde{\mathcal{L}}\widetilde{u}\left(  y,x,t\right)   &  =\mathcal{L}%
u\left(  x,t\right)  \cdot\phi\left(  y\right)  e^{i\sqrt{\lambda}y}+u\left(
x,t\right)  \partial_{yy}^{2}\left(  \phi\left(  y\right)  e^{i\sqrt{\lambda
}y}\right) \\
&  =e^{i\sqrt{\lambda}y}\left\{  \left(  \mathcal{L}u-\lambda u\right)
\left(  x,t\right)  \cdot\phi\left(  y\right)  +u\left(  x,t\right)  \left(
\phi^{\prime\prime}\left(  y\right)  +2i\sqrt{\lambda}\phi^{\prime}\left(
y\right)  \right)  \right\}
\end{align*}
so that%
\begin{equation}
\Vert\widetilde{\mathcal{L}}\widetilde{u}\Vert_{L^{p}\left(  \mathbb{R}%
^{N+2}\right)  }\leq c\left\{  \Vert\mathcal{L}u-\lambda u\Vert_{L^{p}\left(
\mathbb{R}^{N+1}\right)  }+\left(  1+\sqrt{\lambda}\right)  \Vert
u\Vert_{L^{p}\left(  \mathbb{R}^{N+1}\right)  }\right\}  . \label{3}%
\end{equation}
Then (\ref{1})-(\ref{2})-(\ref{3}) imply:%
\begin{equation}%
\begin{split}
&  \left\Vert u\right\Vert _{W_{X}^{2,p}\left(  \mathbb{R}^{N+1}\right)
}+\lambda\left\Vert u\right\Vert _{L^{p}\left(  \mathbb{R}^{N+1}\right)  }\\
&  \qquad\leq c\left\{  \Vert\mathcal{L}u-\lambda u\Vert_{L^{p}\left(
\mathbb{R}^{N+1}\right)  }+\left(  1+\sqrt{\lambda}\right)  \Vert
u\Vert_{L^{p}\left(  \mathbb{R}^{N+1}\right)  }\right\}  .\label{4}%
\end{split}
\end{equation}
Now, there exists $\lambda_{0}>0$ such that
\[
c\left(  1+\sqrt{\lambda}\right)  \leq\frac{\lambda}{2}\text{ for every
}\lambda\geq\lambda_{0}.
\]
For these $\lambda$ and $\lambda_{0}$, and the same $c$ as in (\ref{4}), we
get%
\[
\left\Vert u\right\Vert _{W_{X}^{2,p}\left(  \mathbb{R}^{N+1}\right)  }\leq
c\Vert\mathcal{L}u-\lambda u\Vert_{L^{p}\left(  \mathbb{R}^{N+1}\right)  },
\]
so we are done.
\end{proof}

\subsection{Solvability of $\mathcal{L}u-\lambda u=f$}

We now turn to study the solvability of the equation%
\begin{equation}
\mathcal{L}u-\lambda u=f\quad\text{in $S_{T}$}, \label{eq:Lulambdaf}%
\end{equation}
where $-\infty<T\leq+\infty,\,f\in L^{p}(S_{T})$ and $\lambda\geq\lambda_{0}$
(with $\lambda_{0}>0$ as in Theorem \ref{thm:ImprovedLambda}). As anticipated,
our approach relies on the method of continuity, which allows us to relate the
solvability of \eqref{eq:Lulambdaf} with the solvability of
\begin{equation}
\mathcal{K}u-\lambda u=f\quad\text{in $S_{T}$}, \label{eq:Kulambdaf}%
\end{equation}
where $\mathcal{K}$ is the \emph{constant-coefficient} Kolmogorov operator,
that is,
\begin{equation}
\mathcal{K}=\Delta_{\mathbb{R}^{q}}+Y. \label{eq:opK}%
\end{equation}
Thus, we begin by investigating the solvability of \eqref{eq:Kulambdaf}. To
this end, following the approach by Krylov \cite[Chap. 2]{KrylovBook}, we
first prove the subsequent results.

\begin{proposition}
[Liouville-type property]\label{thm:SMPLiouville}Let $\lambda>0$ be fixed, and
let $\mathcal{K}_{\lambda}=\mathcal{K}-\lambda$ \emph{(}with $\mathcal{K}$ as
in \eqref{eq:opK}\emph{)}. Moreover, let $\mathcal{K}_{\lambda}^{\ast}$ be the
formal adjoint of $\mathcal{K}_{\lambda}$, that is,
\begin{equation}
\label{eq:KlambdastarExplicit}\mathcal{K}_{\lambda}^{\ast}=\Delta
_{\mathbb{R}^{q}}-Y-\lambda.
\end{equation}
If $f\in C^{\infty}(\mathbb{R}^{N+1})\cap L^{\infty}(\mathbb{R}^{N+1})$ is
such that
\[
\mathcal{K}\text{$_{\lambda}^{\ast}f=0$ in $\mathbb{R}^{N+1}$},
\]
then we necessarily have $f\equiv0$ in $\mathbb{R}^{N+1}$.
\end{proposition}

The previous proposition is stated for the adjoint operator $\mathcal{K}%
_{\lambda}^{\ast}$ just because we will need it in this form; the same
property holds for $\mathcal{K}_{\lambda}$ as well. \medskip

In order to prove Proposition \ref{thm:SMPLiouville} we need the following lemma.

\begin{lemma}
\label{lem:funzionemodificata}There exists a function $F:\mathbb{R\rightarrow
R}$ such that $F\in C^{\infty}(\mathbb{R})$, $F\left(  z\right)  =1$ for
$z\in\left[  -1,1\right]  $ and, for some constant $c>0,$%
\[
\left\vert \frac{F^{\prime\prime}(z)}{F(z)}\right\vert \leq c\quad
\text{and}\quad\left\vert \frac{F^{\prime}(z)}{F(z)}\right\vert \leq
c\left\vert z\right\vert \qquad\text{for every $z\in\mathbb{R}$}.
\]

\end{lemma}

\begin{proof}
Let $\phi\in C^{\infty}(\mathbb{R})$ be such that

\begin{itemize}
\item[i)] $0\leq\phi\leq1$ pointwise on $\mathbb{R}$;

\item[ii)] $\phi\equiv0$ on $[-1,1]$ and $\phi\equiv1$ on $\mathbb{R}%
\setminus\lbrack-2,2]$,
\end{itemize}

and let
\[
F(z)=\cosh(z\phi(z)).
\]
Then%
\begin{align*}
F^{\prime}(z)  &  =\sinh(z\phi(z))(\phi(z)+z\phi^{\prime}(z));\\
F^{\prime\prime}(z)  &  =\cosh(z\phi(z))(\phi(z)+z\phi^{\prime2}+\sinh
(z\phi(z))(2\phi^{\prime}(z)+z\phi^{\prime\prime}(z)).
\end{align*}
\vspace*{0.05cm}

From this, using the properties i)\thinspace-\thinspace ii) of $\phi$ and
observing that $\phi^{\prime}(z),\phi^{\prime\prime}(z)\neq0$ only for
$1<|z|<2$, we get
\[
\left\vert \frac{F^{\prime\prime}(z)}{F(z)}\right\vert \leq(\phi
(z)+z\phi^{\prime2}+|\tanh(z\phi(z))|\cdot|2\phi^{\prime}(z)+z\phi
^{\prime\prime}(z)|\leq c.
\]
Moreover, since $\left\vert \tanh(z)\right\vert \leq\left\vert z\right\vert $
for every $z\in\mathbb{R}$, we get
\[
\left\vert \frac{F^{\prime}(z)}{F(z)}\right\vert \leq|\tanh(z\phi
(z))|\cdot|\phi(z)+z\phi^{\prime}(z)|\leq|z|\big(\phi(z)^{2}+\phi
(z)|z\phi^{\prime}(z)|\big)\leq c\left\vert z\right\vert .
\]
This ends the proof.
\end{proof}

With Lemma \ref{lem:funzionemodificata}, we can provide the

\begin{proof}
[Proof of Proposition \ref{thm:SMPLiouville}]We argue by contradiction,
assuming that there exists some $\bar{\xi}\in\mathbb{R}^{N+1}$ such that
$f(\bar{\xi})\neq0$. Following \cite{KrylovBook}, we then fix $\varepsilon
\in(0,1)$ (to be chosen conveniently small later on), and we introduce the
auxiliary function
\[
\zeta(x,t)=F(\varepsilon\hat{p}(x))\cosh(\varepsilon t),
\]
where $F(z)$ is as in Lemma \ref{lem:funzionemodificata}, and $\hat{p}$ is a
homogeneous norm which is globally equivalent to $\Vert\cdot\Vert$, and it is
also smooth outside the origin, e.g.
\[
\hat{p}(x)=\left(  \sum_{i=1}^{N}|x_{i}|^{2q_{N}!/q_{i}}\right)
^{1/(2q_{N}!)}.
\]
Accordingly, we define $g=f/\zeta$.

We now observe that, since $\hat{p}$ is smooth outside the origin and since
$F\equiv1$ in $\left[  -1,1\right]  $, we have $\zeta\in C^{\infty}%
(\mathbb{R}^{N+1})$; hence, the same is true of $g$ (notice that $\zeta\geq
1$). Moreover, since $f$ is bounded, we have
\[
\text{$g(\xi)\rightarrow0$ as $\rho(\xi)=\Vert x\Vert+\sqrt{|t|}%
\rightarrow+\infty$}.
\]
In particular, there exists $\xi_{0}\in\mathbb{R}^{N+1}$ such that
\[
g(\xi_{0})=\max_{\mathbb{R}^{N+1}}g>0.
\]
Finally, since $\mathcal{K}_{\lambda}^{\ast}f=0$ on $\mathbb{R}^{N+1}$, we
have%
\begin{equation}%
\begin{split}
&  0=\mathcal{K}_{\lambda}^{\ast}f=\mathcal{K}_{\lambda}^{\ast}(\zeta
g)=\zeta\mathcal{K}^{\ast}(g)+2\sum_{i=1}^{q}\partial_{x_{i}}\zeta
\partial_{x_{i}}g+c\left(  \xi\right)  g,\\[0.1cm]
&  \text{where $c(\xi)=\mathcal{K}_{\lambda}^{\ast}(\zeta)=\Delta
_{\mathbb{R}^{q}}\zeta-\langle Bx,\nabla\zeta\rangle+\partial_{t}\zeta
-\lambda\zeta$}.
\end{split}
\label{eq:tocontradict}%
\end{equation}
To proceed further we claim that, by choosing $\varepsilon>0$ small enough, we
have
\begin{equation}
\text{$c(\xi)<0$ pointwise on $\mathbb{R}^{N+1}$}. \label{eq:Kstarlambdaneg}%
\end{equation}
Indeed, since the vector fields $\partial_{x_{i}}$ (for $1\leq i\leq q$) and
$Z=\langle Bx,\nabla\rangle$ are ho\-mo\-ge\-neous of degree $1$ and $2$,
respectively (notice that $Z=Y+\partial_{t}$, and both $Y$ and $\partial_{t}$
are homogeneous of degree $2$), and since the norm $\hat{p}$ is $D_{0}%
(\lambda)$-homogeneous of degree $1$, by a direct computation we get:%
\begin{align*}
&  \Delta_{\mathbb{R}^{q}}\zeta=\cosh(\varepsilon t)\Delta_{\mathbb{R}^{q}%
}\big[x\mapsto F\big(\hat{p}(D_{0}(\varepsilon)x)\big)\big]\\
&  \qquad\,\,=\cosh(\varepsilon t)\Delta_{\mathbb{R}^{q}}\big((F\circ\,\hat
{p})\circ D_{0}(\varepsilon)\big)(x)\\
&  \qquad\,\,=\cosh(\varepsilon t)\cdot\varepsilon^{2}\Delta_{\mathbb{R}^{q}%
}(F\circ\,\hat{p})(D_{0}(\varepsilon)x)\\
&  \qquad\,\,=\cosh(\varepsilon t)\big(\varepsilon^{2}(F^{\prime\prime}%
\circ\,\hat{p})|\nabla_{\mathbb{R}^{q}}\hat{p}|^{2}+\varepsilon^{2}(F^{\prime
}\circ\,\hat{p})\Delta_{\mathbb{R}^{q}}\hat{p}\big)(D_{0}(\varepsilon)x)\\
&  \qquad\,\,=\zeta\Big(\varepsilon^{2}\frac{F^{\prime\prime}}{F}%
(\varepsilon\hat{p}(x))|\nabla_{\mathbb{R}^{q}}\hat{p}|^{2}(D_{0}%
(\varepsilon)x)+\varepsilon^{2}\frac{F^{\prime}}{F}(\varepsilon\hat
{p}(x))\Delta_{\mathbb{R}^{q}}\hat{p}(D_{0}(\varepsilon)x)\Big).
\end{align*}
Since $|\nabla_{\mathbb{R}^{q}}\hat{p}|$ and $\Delta_{\mathbb{R}^{q}}\hat{p}$
are $D_{0}\left(  \lambda\right)  $-homogeneous of degrees $0$ and $-1$,
respectively, using the estimates in Lemma \ref{lem:funzionemodificata} we get%
\[
\left\vert \Delta_{\mathbb{R}^{q}}\zeta\right\vert \leq\kappa\zeta
\Big(\varepsilon^{2}+\frac{\varepsilon^{3}\hat{p}(x)}{\hat{p}(D_{0}%
(\varepsilon)x)}\Big)=2\kappa\zeta\varepsilon^{2}.
\]
Next, we compute
\begin{align*}
\left\vert Z\zeta\right\vert  &  =\cosh(\varepsilon t)\cdot\big|Z\big[x\mapsto
F\big(\hat{p}(D_{0}(\varepsilon)x)\big)\big]\big|\\
&  \quad\,\,=\cosh(\varepsilon t)\cdot\big|Z\big((F\circ\,\hat{p})\circ
D_{0}(\varepsilon)\big)(x)\big|\\
&  \quad\,\,=\cosh(\varepsilon t)\cdot\varepsilon^{2}\big|Z(F\circ\,\hat
{p})(D_{0}(\varepsilon)x)\big|\\
&  \quad\,\,=\cosh(\varepsilon t)\cdot\varepsilon^{2}\big|F^{\prime
}(\varepsilon\hat{p}(x))\cdot(Z\hat{p})(D_{0}(\varepsilon)x)\big|\\
&  \quad\,\,=\zeta\varepsilon^{2}\Big|\frac{F^{\prime}}{F}(\varepsilon\hat
{p}(x))\cdot(Z\hat{p})(D_{0}(\varepsilon)x)\Big|\\
&  \qquad(\text{again by the estimates in Lemma \ref{lem:funzionemodificata}%
})\\
&  \quad\,\,\leq\kappa\zeta\cdot\frac{\varepsilon^{3}\hat{p}(x)}{\hat{p}%
(D_{0}(\varepsilon)x)}=\kappa\zeta\varepsilon^{2};\\[0.1cm]
\partial_{t}\zeta &  =\varepsilon\sinh(\varepsilon t)F(\varepsilon\hat
{p}(x))=\zeta\big(\varepsilon\tanh(\varepsilon t)\big);
\end{align*}
for some `universal' constant $\kappa>0$ independent of $\varepsilon$.

From this, since $\tanh(z)\leq1$ for all $z\geq0$ (and since $c(\xi)$ is a
smooth function on the whole of $\mathbb{R}^{N+1}$), we conclude that
\begin{align*}
c(\xi)  &  =\Delta_{\mathbb{R}^{q}}\zeta-\langle Bx,\nabla\zeta\rangle
+\partial_{t}\zeta-\lambda\zeta\\
&  \leq\Delta_{\mathbb{R}^{q}}\zeta+|Z\zeta|+\partial_{t}\zeta-\lambda\zeta\\
&  \leq\zeta\big(3\kappa\varepsilon^{2}+\varepsilon-\lambda\big)<0,
\end{align*}
provided that $\varepsilon>0$ is sufficiently small.

With \eqref{eq:Kstarlambdaneg} at hand, we can finally come to the end of the
proof. In fact, since the point $\xi_{0}$ is a (global) maximum point for $g$,
we clearly have
\[
\mathcal{K}^{\ast}g(\xi_{0})=\Delta_{\mathbb{R}^{q}}g(\xi_{0})-\langle
Bx,\nabla g(\xi_{0})\rangle+\partial_{t}g(\xi_{0})=\Delta_{\mathbb{R}^{q}%
}g(\xi_{0})\leq0;
\]
on the other hand, by combining \eqref{eq:tocontradict}\thinspace
-\thinspace\eqref{eq:Kstarlambdaneg}, we have, since $\partial_{x_{i}}g\left(
x_{0}\right)  =0,$
\[
0=\zeta\mathcal{K}^{\ast}(g)(\xi_{0})+c(\xi_{0})g(\xi_{0})\leq c(\xi_{0}%
)g(\xi_{0})<0
\]
(recall that $g(\xi_{0})>0$), which is clearly a contradiction.
\end{proof}

\begin{proposition}
\label{thm:Density} Let $\lambda>0$ be fixed, and let $\mathcal{K}_{\lambda
}=\mathcal{K}-\lambda$. Then,
\[
\text{$X=\mathcal{K}_{\lambda}(C_{0}^{\infty}(\mathbb{R}^{N+1}))$ is dense in
$L^{p}(\mathbb{R}^{N+1})$}.
\]

\end{proposition}

\begin{proof}
We argue by contradiction, assuming that the vector space $X$ is not dense in
$L^{p}(\mathbb{R}^{N+1})$. Then, by the Hahn-Banach theorem, there exists a
nontrivial linear continuous functional on $L^{p}(\mathbb{R}^{N+1})$ which
vanishes on $X$. Hence, by Riesz' representation theorem, we can find some
non-zero function $f\in L^{q}(\mathbb{R}^{N+1})$ (where $q$ is the conjugate
exponent of $p$, that is, $1/q=1-1/p$) such that
\begin{equation}
\int_{\mathbb{R}^{N+1}}f(\xi)\mathcal{K}_{\lambda}\varphi(\xi)\,d\xi
=0\quad\text{for every $\varphi\in C_{0}^{\infty}(\mathbb{R}^{N+1})$}.
\label{eq:dovepartire}%
\end{equation}
This means that
\begin{equation}
\mathcal{K}\text{$_{\lambda}^{\ast}f=0$ in $\mathcal{D}^{\prime}%
(\mathbb{R}^{N+1})$}; \label{eq:KlambdazeroDistrib}%
\end{equation}
thus, since $\mathcal{K}_{\lambda}^{\ast}$ is $C^{\infty}$-hypoelliptic, we
derive that $f\in C^{\infty}(\mathbb{R}^{N+1})$, and
\eqref{eq:KlambdazeroDistrib} holds pointwise on $\mathbb{R}^{N+1}$. Despite
this fact, we cannot directly apply Proposition \ref{thm:SMPLiouville} to
deduce that $f\equiv0$, since we do not know that $f$ \emph{is bounded}. To
overcome this issue, following the argument in \cite[Prop.\,5.3.9]{BLULibro}
we introduce the mollified function $f_{\varepsilon}$ defined by
\[
f_{\varepsilon}(\xi)=\int_{\mathbb{R}^{N+1}}f(\eta)J_{\varepsilon}(\xi
\circ\eta^{-1})\,d\eta
\]
(with $J_{\varepsilon}$ the usual family of mollifiers `adapted' to the group
$\mathbb{G}$). Then, we easily see that $f_{\varepsilon}\in C^{\infty
}(\mathbb{R}^{N+1})$ and we can write, for every $\phi\in C_{0}^{\infty
}\left(  \mathbb{R}^{N+1}\right)  $,%
\begin{align*}
\int_{\mathbb{R}^{N+1}}f_{\varepsilon}(\xi)\mathcal{K}_{\lambda}\phi\left(
\xi\right)  \,d\xi &  =\int_{\mathbb{R}^{N+1}}\left(  \int_{\mathbb{R}^{N+1}%
}f(\eta^{-1}\circ\xi)J_{\varepsilon}(\eta)\,d\eta\right)  \mathcal{K}%
_{\lambda}\phi\left(  \xi\right)  \,d\xi\\
&  =\int_{\mathbb{R}^{N+1}}J_{\varepsilon}(\eta)\left(  \int_{\mathbb{R}%
^{N+1}}f(\eta^{-1}\circ\xi)\mathcal{K}_{\lambda}\phi\left(  \xi\right)
\,d\xi\right)  d\eta\\
&  =\int_{\mathbb{R}^{N+1}}J_{\varepsilon}(\eta)\left(  \int_{\mathbb{R}%
^{N+1}}f(\xi)\mathcal{K}_{\lambda}\phi\left(  \eta\circ\xi\right)
\,d\xi\right)  d\eta=0
\end{align*}
because, by (\ref{eq:dovepartire}),%
\[
\int_{\mathbb{R}^{N+1}}f(\xi)\mathcal{K}_{\lambda}\phi\left(  \eta\circ
\xi\right)  \,d\xi=\int_{\mathbb{R}^{N+1}}f(\xi)\mathcal{K}_{\lambda}%
\phi^{\eta}\left(  \xi\right)  \,d\xi=0
\]
since the function $\phi^{\eta}\left(  \xi\right)  =\phi\left(  \eta\circ
\xi\right)  $ belongs to $C_{0}^{\infty}(\mathbb{R}^{N+1})$. As a consequence,
we have $\mathcal{K}_{\lambda}^{\ast}(f_{\varepsilon})=0$ in $\mathcal{D}%
^{\prime}(\mathbb{R}^{N+1})$ and thus, by the $C^{\infty}$-hypoellipticity of
$\mathcal{K}_{\lambda}^{\ast}$,
\[
\text{$f_{\varepsilon}\in C^{\infty}(\mathbb{R}^{N+1})$ \quad and
\quad$\mathcal{K}_{\lambda}^{\ast}(f_{\varepsilon})=0$ pointwise in
$\mathbb{R}^{N+1}$}.
\]
On the other hand, since $f\in L^{q}(\mathbb{R}^{N+1})$, we easily see that
the (smooth) function $f_{\varepsilon}$ is bounded (for every fixed
$\varepsilon>0$): in fact, we have
\begin{align*}
\left\vert f_{\varepsilon}(\xi)\right\vert  &  \leq\Vert f\Vert_{L^{q}%
(\mathbb{R}^{N+1})}\left(  \int_{\mathbb{R}^{N+1}}J_{\varepsilon}^{p}(\xi
\circ\eta^{-1})\,d\eta\right)  ^{1/p}\\
&  (\text{setting $\eta=\zeta^{-1}\circ\xi$, and noting that $d\eta=d\zeta$%
})\\
&  =\Vert f\Vert_{L^{q}(\mathbb{R}^{N+1})}\left(  \int_{\mathbb{R}^{N+1}%
}J_{\varepsilon}^{p}(\zeta)\,d\zeta\right)  ^{1/p}=c\left(  \varepsilon
\right)  <+\infty.
\end{align*}
Hence, we can apply Proposition \ref{thm:SMPLiouville}, obtaining
\[
\text{$f_{\varepsilon}\equiv0$ pointwise in $\mathbb{R}^{N+1}$}.
\]
From this, since $f_{\varepsilon}\rightarrow f$ as $\varepsilon\rightarrow
0^{+}$ in $L_{\mathrm{loc}}^{1}(\mathbb{R}^{N+1})$ (see, e.g.,
\cite[Rem.\,5.3.8]{BLULibro}), we conclude that $f\equiv0$ a.e.\thinspace in
$\mathbb{R}^{N+1}$, but this is a contradiction.
\end{proof}

Thanks to the above results, we can then prove the solvability of \eqref{eq:Kulambdaf}.

\begin{theorem}
[Solvability of $\mathcal{K}_{\lambda}u=f$]\label{thm:SolvabilityKlambda} Let
$\lambda\geq\lambda_{0}$ be fixed \emph{(}where $\lambda_{0}>0$ is as in
Theorem \ref{Thm global strip}\emph{)}, and let $-\infty<T\leq+\infty$ and
$p\in\left(  1,\infty\right)  .$ Then, for every $f\in L^{p}(S_{T})$ there
exists a \emph{unique} function $u\in W_{X}^{2,p}(S_{T})$ such that
\begin{equation}
\text{$\mathcal{K}u-\lambda u=f$ a.e.\thinspace in $S_{T}$}.
\label{eq:KulambdaThm}%
\end{equation}
Moreover, we have the estimate
\begin{equation}
\Vert u\Vert_{W_{X}^{2,p}(S_{T})}\leq c\,\Vert f\Vert_{L^{p}(S_{T})},
\label{eq:aprioriestimateKu}%
\end{equation}
with the constant $c$ as in Theorem \ref{thm:ImprovedLambda}.
\end{theorem}

\begin{proof}
First of all we observe that, if a solution $u\in W_{X}^{2,p}(S_{T})$ of
equation \eqref{eq:KulambdaThm} does exist (for some fixed $f\in
L^{p}(\mathbb{R}^{N+1})$), by Theorem \ref{thm:ImprovedLambda} we have
\[
\Vert u\Vert_{W_{X}^{2,p}(S_{T})}\leq c\,\Vert\mathcal{K}u-\lambda
u\Vert_{L^{p}(S_{T})}=c\,\Vert f\Vert_{L^{p}(S_{T})};
\]
this proves at once the validity of the estimate \eqref{eq:aprioriestimateKu},
and the uniqueness part of the theorem.

As for the existence part, we split the proof into two steps. \medskip

\noindent-\thinspace\thinspace\emph{Step I).} Let us first prove the
solvability of \eqref{eq:KulambdaThm} assuming that $T=+\infty$ (that is, when
$S_{T}=\mathbb{R}^{N+1}$). In this case, given any $f\in L^{p}(\mathbb{R}%
^{N+1})$, we know from Pro\-position \ref{thm:Density} that there exists a
sequence $\{\varphi_{n}\}_{n}\subseteq C_{0}^{\infty}(\mathbb{R}^{N+1})$ such
that
\[
f_{n}=(\mathcal{K}-\lambda)\varphi_{n}\rightarrow f\quad\text{as
$n\rightarrow+\infty$ in $L^{p}(\mathbb{R}^{N+1})$};
\]
on the other hand, from Theorem \ref{thm:ImprovedLambda} we infer that
\begin{align*}
\Vert\varphi_{n}-\varphi_{m}\Vert_{W_{X}^{2,p}(\mathbb{R}^{N+1})}  &  \leq
c\,\Vert(\mathcal{K}-\lambda)(\varphi_{n}-\varphi_{m})\Vert_{L^{p}%
(\mathbb{R}^{N+1})}\\
&  =c\,\Vert f_{n}-f_{m}\Vert_{L^{p}(\mathbb{R}^{N+1})},
\end{align*}
and this shows that $\{\varphi_{n}\}_{n}$ is a \emph{Cauchy sequence} in
\emph{the Banach space} $W_{X}^{2,p}(\mathbb{R}^{N+1})$. Thus, we can find
some $u\in W_{X}^{2,p}(\mathbb{R}^{N+1})$ such that
\begin{equation}
\varphi\text{$_{n}\rightarrow u$ as $n\rightarrow+\infty$ in $W_{X}%
^{2,p}(\mathbb{R}^{N+1})$}. \label{eq:untouW2p}%
\end{equation}
Now, since $\mathcal{K}=\Delta_{\mathbb{R}^{q}}+Y$, by \eqref{eq:untouW2p}
(and the definition of $W_{X}^{2,p}$, see Definition \ref{Def Sobolev}) we
infer that $f_{n}=(\mathcal{K}-\lambda)\varphi_{n}\rightarrow(\mathcal{K}%
-\lambda)u$ as $n\rightarrow+\infty$ in $L^{p}(\mathbb{R}^{N+1})$; from this,
since we also have $f_{n}\rightarrow f$ as $n\rightarrow+\infty$ in
$L^{p}(\mathbb{R}^{N+1})$, we conclude that
\[
\text{$\mathcal{K}u-\lambda u=f$ a.e.\thinspace in $\mathbb{R}^{N+1}$},
\]
and this proves that $u$ is a solution of \eqref{eq:KulambdaThm}. \medskip

\noindent-\thinspace\thinspace\emph{Step II).} Let us now turn to prove the
solvability of equation \eqref{eq:KulambdaThm} in the case $T<+\infty$. Let
then $f\in L^{p}(S_{T})$ be fixed, and let
\[
g=f\cdot\mathbf{1}_{S_{T}}.
\]
Since, obviously, we have $g\in L^{p}(\mathbb{R}^{N+1})$, from the previous
\emph{Step I)} we know that there exists a (unique) function $v\in
W^{2,p}(\mathbb{R}^{N+1})$ such that
\[
\mathcal{K}v-\lambda v=g\quad\text{a.e.\thinspace in $\mathbb{R}^{N+1}$}.
\]
Then, setting $u=v|_{S_{T}}$, we immediately see that $u\in W^{2,p}(S_{T})$,
and that $u$ is a solution of equation \eqref{eq:KulambdaThm} (since $g=f$
a.e.\thinspace in $S_{T}$). This ends the proof.
\end{proof}

By combining Theorem \ref{thm:SolvabilityKlambda} with the localized estimates
in Theorem \ref{Thm global strip}, we can now prove the solvability of \eqref{eq:Lulambdaf}.

\begin{theorem}
[Solvability of $\mathcal{L}u-\lambda u=f$]\label{thm:SolvabilityLlambda} Let
$\lambda\geq\lambda_{0}$ be fixed \emph{(}where $\lambda_{0}>0$ is as in
Theorem \ref{Thm global strip}\emph{)}, and let $-\infty<T\leq+\infty$.
Moreover, let $\mathcal{L}$ be an operator as in \eqref{L}, assume that
\emph{\textbf{(H1)}, \textbf{(H2)}, \textbf{(H3)}} hold and let $p\in\left(
1,\infty\right)  $. Then, for every $f\in L^{p}(S_{T})$ there exists a
\emph{unique} function $u\in W_{X}^{2,p}(S_{T})$ such that
\begin{equation}
\text{$\mathcal{L}u-\lambda u=f$ a.e.\thinspace in $S_{T}$}.
\label{eq:LulambdaThm}%
\end{equation}
Moreover, we have the estimate
\begin{equation}
\Vert u\Vert_{W_{X}^{2,p}(S_{T})}\leq c\,\Vert f\Vert_{L^{p}(S_{T})}
\label{eq:aprioriestimateLu}%
\end{equation}
with the constant $c$ as in Theorem \ref{thm:ImprovedLambda}.
\end{theorem}

\begin{proof}
We are going to prove the theorem by using the method of conti\-nuity. To this
end we consider, for every fixed $r\in\lbrack0,1]$, the operator
\[
\mathcal{P}_{r}:W_{X}^{2,p}(S_{T})\rightarrow L^{p}(S_{T}),\qquad
\mathcal{P}_{r}u=\mathcal{L}_{r}u-\lambda u,
\]
where $\mathcal{L}_{r}=(1-r)\mathcal{K}+r\mathcal{L}$, that is,
\[
\mathcal{L}_{r}=\sum_{i,j=1}^{q}\big((1-r)\delta_{ij}+ra_{ij}%
(x,t)\big)\partial_{x_{i}x_{j}}^{2}+Y\equiv\sum_{i,j=1}^{q}a_{ij}%
^{r}(x,t)\partial_{x_{i}x_{j}}^{2}+Y.
\]
We then observe that, since the coefficients $a_{ij}$ are globally bounded in
$\mathbb{R}^{N+1}$ (see assumption \textbf{(H1)}), and since $0\leq r\leq1$,
for every $u\in W_{X}^{2,p}(S_{T})$ we get
\begin{align*}
\Vert\mathcal{P}_{r}u\Vert_{L^{p}(S_{T})}  &  \leq\sum_{i,j=1}^{q}\Vert
a_{ij}^{r}\partial_{x_{i}x_{j}}^{2}u\Vert_{L^{p}(S_{T})}+\Vert Yu\Vert
_{L^{p}(S_{T})}+\lambda\Vert u\Vert_{L^{p}(S_{T})}\\
&  \leq(1+\max_{i,j}\Vert a_{ij}\Vert_{L^{\infty}(S_{T})})\sum_{i,j=1}%
^{q}\Vert\partial_{x_{i}x_{j}}^{2}u\Vert_{L^{p}(S_{T})}\\
&  \qquad+\Vert Yu\Vert_{L^{p}(S_{T})}+\lambda\Vert u\Vert_{L^{p}(S_{T})}\\
&  \leq c\,\Vert u\Vert_{W_{X}^{2,p}(S_{T})},
\end{align*}
and this proves that $\mathcal{P}_{r}$ is a (linear and) continuous operator
from $W^{2,p}(S_{T})$ into $L^{p}(S_{T})$, with norm bounded independently of
$r$. Moreover, by Theorem \ref{thm:ImprovedLambda} there exists a constant
$c>0$, independent of $r$, such that
\begin{equation}
\Vert u\Vert_{W_{X}^{2,p}(S_{T})}\leq c\,\Vert\mathcal{P}_{r}u\Vert
_{L^{p}(S_{T})}. \label{eq:aprioriLr}%
\end{equation}
We stress that we are entitled to apply Theorem \ref{thm:ImprovedLambda} to
the ope\-rator $\mathcal{P}_{r}$, since it is a KFP operator of the form
\eqref{L}, satisfying assumptions \textbf{(H1)}, \textbf{(H2)}, \textbf{(H3)},
with uniform constants as $r$ ranges in $\left[  0,1\right]  $. We can then
apply the method of continuity: since we know from Theorem
\ref{thm:SolvabilityKlambda} that the operator $\mathcal{P}_{0}=\mathcal{K}%
-\lambda$ is surjective, we conclude that the same is true of the operator
$\mathcal{P}_{1}=\mathcal{L}-\lambda$, that is, for every fixed $f\in
L^{p}(\mathbb{R}^{N+1})$ there exists a unique function $u\in W_{X}%
^{2,p}(S_{T})$ such that
\[
\text{$\mathcal{L}u-\lambda u=f$ a.e.\thinspace in $S_{T}$}.
\]
Finally, the validity of estimate \eqref{eq:aprioriestimateLu} follows from
\eqref{eq:aprioriLr} (with $r=1$).
\end{proof}

\subsection{\noindent The Cauchy problem for $\mathcal{L}$%
\label{Subsec Cauchy}}

We can now prove our well-posedness result for the $\mathcal{L}$\,-\,Cauchy
problem, that is, Theorem \ref{thm:ExistenceCauchy}; we refer to Section
\ref{sec assumptions} for the definition of \emph{solution of the Cauchy
problem}, and for the definition of the involved function spaces $\mathring
{W}_{X}^{2,p}\left(  \Omega_{T}\right)  ,W_{X}^{2,p}(\mathbb{R}^{N})$.

\bigskip

\begin{proof}
[Proof of Theorem \ref{thm:ExistenceCauchy}]We split the proof into three
steps. \medskip

\noindent-\thinspace\thinspace\emph{Step I).} Here we prove the
\emph{existence part} of the theorem, together with the validity of estimate
\eqref{eq:estimateSolCauchy}, in the special case $g=0$.

To this end, we fix $f\in L^{p}(\Omega_{T})$, and we set
\[
h=e^{-\lambda_{0}t}f\cdot\mathbf{1}_{[0,T]},
\]
where $\lambda_{0}>0$ is as in Theorem \ref{thm:SolvabilityLlambda}. (As
usual, the symbol $f\cdot\mathbf{1}_{[0,T]}$ means that the function $f$ has
been defined in the whole strip $S_{T}$, letting it equal to zero for $t<0$).
Since, obviously, $h\in L^{p}(S_{T})$, by Theorem \ref{thm:SolvabilityLlambda}
we know that there exists a (unique) $v\in W_{X}^{2,p}(S_{T})$ such that
\[
\text{$\mathcal{L}v-\lambda_{0}v=h$ a.e.\thinspace in $S_{T}$};
\]
moreover, by applying the estimate (\ref{eq:aprioriestimateLu}) on the strip
$S_{0}$ (to the `restricted' function $v|_{S_{0}}\in W_{X}^{2,p}(S_{0})$), we
have
\[
\Vert v\Vert_{W^{2,p}(S_{0})}\leq c\,\Vert\mathcal{L}v-\lambda_{0}%
v\Vert_{L^{p}(S_{0})}=\Vert h\Vert_{L^{p}(S_{0})}=0,
\]
from which we derive that
\begin{equation}
\text{$v=0$ a.e.\thinspace in $S_{0}$}. \label{eq:vvanishS0}%
\end{equation}
Then, we set $u=e^{\lambda_{0}t}v$, and we claim that $u$ is a solution to the
Cauchy problem \eqref{eq:PbCauchyL}. Indeed, since $v\in W_{X}^{2,p}(S_{T})$,
also $u\in W_{X}^{2,p}(S_{T})$ (here it is important that $T<\infty$), and by
(\ref{eq:vvanishS0}) we conclude that
\[
\text{$u\in$}\mathring{W}_{X}^{2,p}\left(  \Omega_{T}\right)  .
\]
Moreover, since $\mathcal{L}v-\lambda_{0}v=h$, by a direct computation we
have
\begin{align*}
\mathcal{L}u  &  =\mathcal{L}(e^{\lambda_{0}t}v)=e^{\lambda_{0}t}\left(
\mathcal{L}v-\lambda_{0}v\right)  =e^{\lambda_{0}t}h\\
&  (\text{by definition of $h$})\\
&  =f\cdot\mathbf{1}_{[0,T]}=f\quad\text{a.e.\thinspace in $\Omega_{T}$},
\end{align*}
and this proves that $u$ is a solution of the Cauchy problem
\eqref{eq:PbCauchyL}, as claimed.

As for the validity of estimate (\ref{eq:estimateSolCauchy}) it suffices to
note that, by applying (\ref{eq:aprioriestimateLu}) on the strip $S_{T}$ to
$v$ (and exploiting again the finiteness of $T$), we get
\begin{align*}
\Vert u\Vert_{W_{X}^{2,p}(\Omega_{T})}  &  =\,\Vert e^{\lambda_{0}t}%
v\Vert_{W_{X}^{2,p}(\Omega_{T})}\leq c\left(  T\right)  \,\Vert v\Vert
_{W_{X}^{2,p}(S_{T})}\\
&  \leq c\,\Vert\mathcal{L}v-\lambda_{0}v\Vert_{L^{p}(S_{T})}=c\,\Vert
h\Vert_{L^{p}(S_{T})}\\
&  =c\,\Vert e^{-\lambda_{0}t}f\Vert_{L^{p}(\Omega_{T})}\leq c\,\Vert
f\Vert_{L^{p}(\Omega_{T})},
\end{align*}
for some absolute constant $c>0$ (possibly different from line to
line).\medskip

\noindent-\thinspace\thinspace\emph{Step II).} We now consider general data
$f\in L^{p}\left(  \Omega_{T}\right)  $ and $g\in W_{X}^{2,p}(\mathbb{R}^{N})$
for the Cauchy problem. The function $\widetilde{f}=f-\mathcal{L}g$ clearly
belongs to $L^{p}\left(  \Omega_{T}\right)  $; hence, by \emph{Step I)}, we
can find $v\in\mathring{W}_{X}^{2,p}(\Omega_{T})$ such that $v$ is a solution
to%
\begin{equation}
\left\{
\begin{tabular}
[c]{ll}%
$\mathcal{L}v=f-\mathcal{L}g$ & $\text{in }\Omega_{T}$\\
$v\left(  \cdot,0\right)  =0$ & $\text{in }\mathbb{R}^{N}$%
\end{tabular}
\right.  \label{Cauchy_aux}%
\end{equation}
and
\begin{equation}
\Vert v\Vert_{W_{X}^{2,p}(\Omega_{T})}\leq c\,\left\Vert \widetilde{f}%
\right\Vert _{L^{p}(\Omega_{T})}\leq c\,\left\{  \left\Vert f\right\Vert
_{L^{p}(\Omega_{T})}+\Vert g\Vert_{W_{X}^{2,p}(\mathbb{R}^{N})}\right\}  .
\label{apriori_Cauchy1}%
\end{equation}
Define $u\left(  x,t\right)  =v\left(  x,t\right)  +g\left(  x\right)  $. Then
$u\in W_{X}^{2,p}(\Omega_{T})$ and $\mathcal{L}u=f$ in $\Omega_{T}$. Moreover,
$u-g=v\in\mathring{W}_{X}^{2,p}(\Omega_{T})$, hence $u$ is a solution to the
Cauchy problem (\ref{eq:PbCauchyL}).

Finally,%
\begin{equation}
\Vert u\Vert_{W_{X}^{2,p}(\Omega_{T})}\leq\Vert v\Vert_{W_{X}^{2,p}(\Omega
_{T})}+\Vert g\Vert_{W_{X}^{2,p}(\Omega_{T})}. \label{apriori_Cauchy2}%
\end{equation}
By (\ref{g spazio e tempo}), (\ref{apriori_Cauchy1}), and
(\ref{apriori_Cauchy2}), we get (\ref{eq:estimateSolCauchy}).\medskip

\noindent-\thinspace\thinspace\emph{Step III).} We now prove the
\emph{uniqueness part} of the theorem. This \emph{does not follow} directly
from estimate (\ref{eq:estimateSolCauchy}), since this estimate has been
proved \emph{only for the solution $u$} constructed in the previous
\emph{Steps I-II)}. However, let $u_{1},u_{2}\in W_{X}^{2,p}(\Omega_{T})$ be
two solutions of the Cauchy problem (\ref{eq:PbCauchyL}). Then the function
$\bar{u}=u_{1}-u_{2}$ belongs to $\mathring{W}_{X}^{2,p}(\Omega_{T})$ and
solves $\mathcal{L}\overline{u}=0$ in $\Omega_{T}$. Hence also the function%
\[
v=e^{-\lambda_{0}t}\bar{u}%
\]
belongs to $\mathring{W}_{X}^{2,p}(\Omega_{T})$. Moreover, by proceeding as in
\emph{Step I)}, we have
\[
\mathcal{L}v-\lambda_{0}v=e^{-\lambda_{0}t}\mathcal{L}\bar{u}=0\quad
\text{a.e.\thinspace in $\Omega_{T}$}.
\]
As a consequence, by applying (\ref{eq:LpOmogenee}) to the strip $S_{T}$ and
to the function $v$, we obtain the following estimate (recall that $v=0$
a.e.\thinspace in $S_{0}$)
\[
\Vert v\Vert_{W_{X}^{2,p}(S_{T})}\leq c\,\Vert\mathcal{L}v-\lambda_{0}%
v\Vert_{L^{p}(S_{T})}=c\,\Vert\mathcal{L}v-\lambda_{0}v\Vert_{L^{p}(\Omega
_{T})}=0.
\]
This proves that $v=0$ a.e.\thinspace in $S_{T}$, and thus
\[
\text{$u_{1}=u_{2}$ a.e.\thinspace in $S_{T}$}.
\]
This ends the proof.
\end{proof}

\begin{remark}
\label{rem:FiniteT} It is clear from the above proof of Theorem
\ref{thm:ExistenceCauchy} that, differently from the previous results, in this
theorem we have to assume $0<T<+\infty$.

Indeed, if $T = +\infty$ (so that $\Omega_{T} = \mathbb{R}^{N}\times
(0,+\infty)$), the change of variable
\[
v\longleftrightarrow e^{\lambda_{0}t}v=u
\]
(used in \emph{Step I)} to link the solvability of the Cauchy problem to that
of the e\-qua\-tion $\mathcal{L}v-\lambda_{0}v = h$) \emph{does not preserve}
the property of belonging to $W^{2,p}_{X}(\Omega_{T})$.
\end{remark}

\bigskip

\noindent\textbf{Acknowledgements. }{The Authors are members of the research
group \textquotedblleft Grup\-po Na\-zio\-nale per l'Analisi Matematica, la
Probabilit\`{a} e le loro Applicazioni\textquotedblright\ of the Italian
\textquotedblleft Istituto Na\-zio\-na\-le di Alta
Matematica\textquotedblright. The first Author is partially supported by the
PRIN 2022 project 2022R537CS \emph{$NO^{3}$ - Nodal Optimization, NOnlinear
elliptic equations, NOnlocal geometric problems, with a focus on regularity},
funded by the European Union - Next Generation EU}; the second Author is
partially supported by the PRIN 2022 project \emph{Partial differential
equations and related geometric-functional inequalities}, financially
supported by the EU, in the framework of the \textquotedblleft Next Generation
EU initiative\textquotedblright.

\bigskip

\noindent\textbf{Address. }S. Biagi, M. Bramanti:

\noindent Dipartimento di Matematica. Politecnico di Milano. \newline\noindent
Via Bonardi 9. 20133 Milano. Italy.

\noindent e-mail: stefano.biagi@polimi.it; marco.bramanti@polimi.it

\end{document}